\documentclass[secthm,seceqn,amsthm,ussrhead,reqno, 12pt]{amsart}
\usepackage[utf8]{inputenc}
\usepackage[english]{babel}
\usepackage[symbol]{footmisc}
\usepackage{amssymb,amsmath,amsthm,amsfonts,xcolor,enumerate,hyperref,comment,longtable,cleveref}

\usepackage{times}
\usepackage{cite}
\usepackage{pdflscape}
\usepackage{ulem}
\usepackage[mathcal]{euscript}
\usepackage{tikz}
\usepackage{hyperref}
\usepackage{cancel}
\usepackage{stmaryrd}

\usetikzlibrary{arrows}

\newtheorem{theorem}{Theorem}
\newtheorem{lemma}[theorem]{Lemma}

\newtheorem{definition}[theorem]{Definition}

\newtheorem{openq}{Open question}

\newtheorem*{theoremA}{Theorem A}

\newtheorem*{theoremB}{Theorem B}
\newtheorem*{theoremC}{Theorem C}
\newtheorem*{theoremD}{Theorem D}
\newtheorem*{theoremE}{Theorem E}
\newtheorem*{theoremF}{Theorem F}

\newtheorem*{definitionA}{Definition}

\usepackage{stmaryrd}
\usepackage{xcolor}

\begin{document}

\noindent{\Large
The algebraic and geometric classification of  \\  nilpotent binary and mono Leibniz   algebras}\footnote{
The first part of this work is supported by
FCT   UIDB/MAT/00212/2020 and UIDP/MAT/00212/2020;
 grant FZ-202009269, Ministry of higher education, science and innovations of the Republic of Uzbekistan.
The second part of this work is supported by the Russian Science Foundation under grant 22-71-10001.
}
\footnote{Corresponding author: Ivan Kaygorodov   (kaygorodov.ivan@gmail.com)}

 \bigskip

 \bigskip

\begin{center}

 {\bf
Kobiljon Abdurasulov\footnote{CMA-UBI, Universidade da Beira Interior, Covilh\~{a}, Portugal; Institute of Mathematics Academy of
Sciences of Uzbekistan, Tashkent, Uzbekistan; \ abdurasulov0505@mail.ru},
Ivan Kaygorodov\footnote{CMA-UBI, Universidade da Beira Interior, Covilh\~{a}, Portugal;   \
Moscow Center for Fundamental and Applied Mathematics, Moscow,   Russia; \
Saint Petersburg  University, Russia; \
    kaygorodov.ivan@gmail.com}
  \&
Abror Khudoyberdiyev\footnote{Institute of Mathematics Academy of
Sciences of Uzbekistan, Tashkent, Uzbekistan; \
National University of Uzbekistan, Tashkent, Uzbekistan, \ khabror@mail.ru}

}

\end{center}

 \bigskip

\

\noindent{\bf Abstract}:
{\it This paper is devoted to the complete algebraic and geometric classification of complex $5$-dimensional nilpotent binary   Leibniz and $4$-dimensional nilpotent mono Leibniz algebras.
As a corollary, we have the complete algebraic and geometric classification of complex $4$-dimensional nilpotent algebras of nil-index $3$.
 }

\medskip

\noindent {\bf Keywords}:
{\it Leibniz algebra,
binary Leibniz algebra,
mono Leibniz algebras,
nilpotent algebra,
algebraic classification,
geometric classification.}

\

 \bigskip

\noindent {\bf MSC2020}:  17A30, 17A32, 14D06, 14L30.

 \bigskip

\section*{Introduction}

The algebraic classification (up to isomorphism) of algebras of dimension $n$ from a certain variety
defined by a certain family of polynomial identities is a classic problem in the theory of non-associative algebras.
There are many results related to the algebraic classification of small-dimensional algebras in the varieties of
Jordan, Lie, Leibniz, Zinbiel, and many other algebras \cite{      DGK18,  kkl20,   GRH3, shaf,   kkp20} and references in \cite{k23,l24}.
 Geometric properties of a variety of algebras defined by a family of polynomial identities have been an object of study since 1970's (see, \cite{ wolf1,     KW01,      kkp20,       gabriel,    ckls,     GRH, GRH2,   GRH3,  fkkv,    kppv }  and references in \cite{k23}).
 Gabriel described the irreducible components of the variety of $4$-dimensional unital associative algebras~\cite{gabriel}.
 Grunewald and O'Halloran  calculated the degenerations for the variety of $5$-dimensional nilpotent Lie algebras~\cite{GRH}.
Degenerations have also been used to study a level of complexity of an algebra~\cite{wolf1}.
 The study of degenerations of algebras is very rich and closely related to deformation theory, in the sense of Gerstenhaber \cite{ger63}.

\newpage

If $\Omega$ is a variety of algebras defined by a family of polynomial identities,
then we say that an algebra  ${\rm A} \in \Omega_i$ if and only if each $i$-generated subalgebra of ${\rm A}$ gives an algebra from $\Omega.$
In particular,
if ${\rm  A} \in \Omega_1,$ then ${\rm A}$ is a mono-$\Omega$ algebra,
if ${\rm A} \in \Omega_2,$ then ${\rm A}$ is a binary-$\Omega$ algebra.
For example, let ${\rm Ass}$ be the class of associative algebras, then by Artin’s theorem, the class ${\rm Ass}_2$
coincides with the class of alternative algebras.
It follows from Albert’s Theorem that the class ${\rm Ass}_1$ coincides with
the class of power-associative algebras \cite{albert}.
It is easy to see that ${\rm  Lie}_1$ coincides with anticommutative algebras,
i.e., they satisfy the identity $x^2=0$ and the identities of ${\rm Lie}_2$ are described by Gainov \cite{Idbilie}.
The main non-trivial example of binary Lie algebras is the class of Malcev algebras, defined by Malcev in \cite{malcev55}.
The algebraic theory of binary Lie algebras was developed in some papers by Kuzmin, Filippov, and Grishkov
(see, for example, \cite{k67,G81,F91} and references therein).
So,
Kuzmin proved Engel's theorem for binary Lie algebras \cite{k67}.
Grishkov   described  complex semisimple finite-dimensional binary-Lie algebras \cite{G81}
Filippov characterized prime binary-Lie algebras \cite{F91}.
Umirbaev proved that the variety of complex metabelian binary-Lie algebras  is Spechtian (i.e., every subvariety of it has a finite basis of identities) in \cite{u84}.
Chupina proved that two complex finite-dimensional semisimple binary Lie algebras are isomorphic if their lattices of subalgebras are  isomorphic \cite{C91}.
The question of specialty of binary-Lie algebras is considered in \cite{as11}.
 Arenas and  Arenas-Carmona studied
the universal Poisson envelope for binary-Lie algebras in \cite{aa13}.
On the other hand, the theory of binary $(-1,1)$-algebras is also under intensive consideration
(the variety of $(-1,1)$-algebras is also known as Lie-admissible right alternative algebras).
So,
the identities of binary $(-1,1)$-algebras are described by Kleinfeld, Smith and Pchelintsev in \cite{p76,ks75}.
Later they developed the theory of binary $(-1,1)$-algebras. For example,
 Pchelintsev proved that  if a complex binary $(-1,1)$-algebra  is a nil algebra of bounded index, then it is locally nilpotent \cite{p91} and described
 irreducible binary $(-1,1)$-bimodules over simple finite-dimensional algebras \cite{p06}.
Hentzel and  Smith proved that each complex simple binary $(-1,1)$ nil algebra is associative \cite{HS86}.
Recently,
defining identities for  mono and binary Zinbiel algebras are described in \cite{ims21}
and defining identities for mono and binary Leibniz algebras are described in \cite{binleib,
g10}.
On the other hand, the
defining identities for mono  symmetric Zinbiel   and mono symmetric Leibniz algebras coincide with nil-algebras of nil-index 3
\cite{bkm22}.

An algebra ${\bf A}$ is called a  Leibniz  algebra  if it satisfies the identity
\begin{center}
    $(xy)z=(xz)y+x(yz).$
\end{center}
Leibniz algebras present a "non antisymmetric" generalization of Lie algebras.
It appeared in some papers of Bloh [in the 1960s] %, Kontsevich  [in 1980s]
and Loday  [in 1990s].
Recently, they appeared in many geometric and physics applications (see, for example,
\cite{       ms22,    m22, STZ21,  aks,  strow20} and references therein).
A systematic study of algebraic properties of Leibniz algebras is started from the Loday paper.
So, several classical theorems from Lie algebras theory have been extended
to the Leibniz algebras case;
many classification results regarding nilpotent, solvable, simple, and semisimple Leibniz algebras
are obtained
(see, for example, \cite{ck13,    CKLO13, akyu, aors21,   kppv,  m22, kky22, T21, STZ21} and references therein).
Leibniz algebras is a particular case of terminal algebras and, on the other hand,
symmetric Leibniz algebras are Poisson admissible algebras. In the present paper, based on a known classification  of $5$-dimensional nilpotent Leibniz algebras \cite{leib5}, we give the algebraic  and geometric classification of
complex $5$-dimensional nilpotent  binary Leibniz and complex $4$-dimensional mono  Leibniz algebras.

\newpage
Our method for classifying nilpotent  Leibniz algebras is based on the calculation of central extensions of nilpotent algebras of smaller dimensions from the same variety.
The algebraic study of central extensions of   algebras has been an important topic for years \cite{  hac16,  ss78}.
First, Skjelbred and Sund used central extensions of Lie algebras to obtain a classification of nilpotent Lie algebras  \cite{ss78}.
Note that the Skjelbred-Sund method of central extensions is an important tool in the classification of nilpotent algebras.
Using the same method,
 $4$-dimensional nilpotent
(bicommutative, commutative, terminal,   and so on) algebras,
 $5$-dimensional nilpotent
(Zinbiel,
symmetric Leibniz  and so on) algebras,
 $6$-dimensional nilpotent
(anticommutative, binary Lie \cite{ack}, Lie, Tortkara and so on) algebras,
 $8$-dimensional   dual Mock-Lie algebras \cite{ckls},
and some others have been described. Our main results related to the algebraic classification of cited varieties are summarized below.

\begin{theoremA}
Up to isomorphism, there are infinitely many isomorphism classes of complex $5$-dimensional nilpotent (non-Leibniz) binary Leibniz algebras,
described explicitly in  Section \ref{secteoA}
in terms of $1$ one-parameter family and $13$ additional isomorphism classes.
\end{theoremA}

\begin{theoremB}
Up to isomorphism, there are infinitely many  complex $4$-dimensional nilpotent (non-binary Leibniz) mono Leibniz algebras,
described explicitly in  Section \ref{secteoB}
in terms of $10$ one-parameter families and $12$ additional isomorphism classes.

\end{theoremB}

\begin{theoremC}
Up to isomorphism, there are infinitely many  complex $4$-dimensional nilpotent (non-$2$-step nilpotent) algebras of nil-index $3$,
described explicitly in  Section \ref{secteoC}
in terms of $4$ one-parameter families and $11$ additional isomorphism classes.

\end{theoremC}

 The degenerations between the (finite-dimensional) algebras from a certain variety $\mathfrak{V}$ defined by a set of identities have been actively studied in the past decade.
The description of all degenerations allows one to find the so-called rigid algebras and families of algebras, i.e. those whose orbit closures under the action of the general linear group form irreducible components of $\mathfrak{V}$
(with respect to the Zariski topology).
We list here some works in which the rigid algebras of the varieties of
all $4$-dimensional nilpotent commutative algebras,
all $6$-dimensional nilpotent anticommutative algebras,
all $8$-dimensional dual Mock Lie algebras \cite{ckls}
have been found.
A full description of degenerations has been obtained
for $2$-dimensional algebras,
for $3$-dimensional anticommutative algebras,
for $3$-dimensional Leibniz algebras,
for $4$-dimensional Zinbiel and  $4$-dimensional nilpotent Leibniz algebras in \cite{kppv},
for $6$-dimensional nilpotent Lie algebras in \cite{GRH},
for $8$-dimensional $2$-step nilpotent anticommutative algebras
and so on.
Our main results related to the geometric classification of cited varieties are summarized below.

\begin{theoremD}
The variety of complex $5$-dimensional nilpotent  binary Leibniz algebras has
dimension {\it 24 }  and it has
10  irreducible components
(in particular, there is only one rigid algebra in this variety).
\end{theoremD}

\begin{theoremE}
The variety of complex $4$-dimensional nilpotent  algebras of nil-index $3$ has
dimension {\it 15}  and it has
2  irreducible components
(in particular, there are no rigid algebras in this variety).
\end{theoremE}

\begin{theoremF}
The variety of complex $4$-dimensional nilpotent mono Leibniz algebras has
dimension {\it 15}  and it has
3  irreducible components
(in particular, there is only one rigid algebra in this variety).
\end{theoremF}

\newpage

\section{The algebraic classification of nilpotent binary and mono  Leibniz algebras}

\subsection{Preliminaries and basic definitions}
Further we use the notation
\begin{longtable}{lcl}
${\mathcal L}(x, y, z)$&$=$&$ x(yz) - (xy)z + (xz)y,$\\
${\mathcal L_{\theta}}(x, y, z)$ & $=$&$\theta(x,yz)-\theta(xy,z)+\theta(xz,y),$
\end{longtable}
where $\theta  \colon {\bf A} \times {\bf A} \longrightarrow {\mathbb V},$ ${\mathbb V}$ is a vector space over ${\mathbb C}$, and $\bf A$ is an algebra over  $\mathbb{C}$.

\begin{definitionA}
A  complex vector space  is called a  Leibniz algebra if it satisfies ${\mathcal L}(x,y,z)=0.$
A  complex vector space  is called a  binary Leibniz algebra if every two-generated subalgebra is a Leibniz algebra.
A  complex vector space  is called a  mono Leibniz algebra if every one-generated subalgebra is a Leibniz algebra.
\end{definitionA}

The algebra ${\bf A}$ is binary Leibniz if and only if it satisfies the identities
\begin{longtable}{lclrcr}
${\mathcal L}(x,y,z)+{\mathcal L}(x,z,y)$&$=$&$0,$
&${\mathcal L}(x,y,z)+{\mathcal L}(z,y,x)$&$=$&$0,$\\
\multicolumn{4}{c}{${\mathcal L}(xy,z,t)+{\mathcal L}(xt,z,y)+{\mathcal L}(zy,x,t)+{\mathcal L}(zt,x,y)$} &$=$&$ 0.$ \\
\end{longtable}

The algebra ${\bf A}$ is mono  Leibniz if and only if it satisfies the identities
\begin{longtable}{lcllcl}
  ${\mathcal L}(a,a,a)$&$=$&$0,$ \quad ${\mathcal L}(a,a,aa)$&$=$&$0.$
\end{longtable}

By    linearizing  these identities,  we have

\begin{longtable}{lcl}
${\mathcal L}(x,y,z)+{\mathcal L}(y,x,z)+{\mathcal L}(y,z,x)+{\mathcal L}(x,z,y)+{\mathcal L}(z,x,y)+{\mathcal L}(z,y,x)$&$=$&$0,$\\
${\mathcal L}(x,y,zt)+{\mathcal L}(x,y,tz)+{\mathcal L}(x,z,yt)+{\mathcal L}(x,t,yz)+{\mathcal L}(x,z,ty)+{\mathcal L}(x,t,zy)$&$+$&\\
${\mathcal L}(y,x,zt)+{\mathcal L}(y,x,tz)+{\mathcal L}(z,x,yt)+{\mathcal L}(t,x,yz)+{\mathcal L}(z,x,ty)+{\mathcal L}(t,x,zy)$&$+$&\\
${\mathcal L}(y,z,xt)+{\mathcal L}(y,t,xz)+{\mathcal L}(z,y,xt)+{\mathcal L}(t,y,xz)+{\mathcal L}(z,t,xy)+{\mathcal L}(t,z,xy)$&$+$&\\
${\mathcal L}(y,z,tx)+{\mathcal L}(y,t,zx)+{\mathcal L}(z,y,tx)+{\mathcal L}(t,y,zx)+{\mathcal L}(z,t,yx)+{\mathcal L}(t,z,yx)$&$=$&$0.$
\end{longtable}

From the definition of binary and mono  Leibniz algebras we can conclude the following:
\begin{enumerate}
  \item \emph{For binary Leibniz algebras}:
  \begin{itemize}
  \item There are no nontrivial $1$-dimensional nilpotent binary Leibniz algebras.
  \item Two-dimensional and three-dimensional nilpotent binary Leibniz algebras are Leibniz algebras.
  \item Two-generated binary Leibniz algebras are Leibniz algebra.
  \item A binary Leibniz algebra $\mathfrak L,$ such that for $\mathfrak L^3=0,$ is a Leibniz algebra.
  \end{itemize}
Thus, non-Leibniz binary Leibniz algebras should be at least three generated.
Consequently, we have that any nilpotent binary Leibniz algebra with a dimension less than five is a Leibniz algebra.

  \item \emph{For mono  Leibniz algebras}
  \begin{itemize}
  \item There are no nontrivial $1$-dimensional nilpotent mono  Leibniz algebras.
  \item One-generated mono  Leibniz algebras are  Leibniz algebras.
  \item Two and three-dimensional nilpotent mono  Leibniz algebras are  Leibniz algebras.
   \item A mono  Leibniz algebra $\mathfrak L,$ such that for $\mathfrak L^3=0,$  is a Leibniz algebra.
   \end{itemize}

 Thus, we conclude that any  nilpotent non-Leibniz mono  Leibniz algebra has at least two generators and $\mathfrak L^3\neq0$.
 \end{enumerate}

In this work, we classify five-dimensional nilpotent non-Leibniz   binary Leibniz algebras and  four-dimensional nilpotent non-Leibniz   mono  Leibniz algebras.

\subsection{Method of classification of nilpotent algebras}
Throughout this paper, we use the notations and methods well written in \cite{hac16},
which we have adapted for the binary Leibniz and mono  Leibniz algebras case with some modifications.
Further in this section, we give some important definitions.

As we know, the central extension is formed by the second cohomology space of a given algebras. We now define the  second cohomology space for binary and unary Leibniz algebras.

\begin{itemize}
  \item Let $({\bf A}, \cdot)$ be a binary Leibniz algebra over  $\mathbb C$
and $\mathbb V$ be a vector space over ${\mathbb C}$. The $\mathbb C$-linear space ${\rm Z}_{\rm BL}^{2}\left(
\bf A,\mathbb V \right) $ is defined as the set of all  bilinear maps $\theta  \colon {\bf A} \times {\bf A} \longrightarrow {\mathbb V}$ such that
\begin{longtable}{lclrcr}
${\mathcal L_{\theta}}(x,y,z)+{\mathcal L_{\theta}}(x,z,y)$&$=$&$0,$
&${\mathcal L_{\theta}}(x,y,z)+{\mathcal L_{\theta}}(z,y,x)$&$=$&$0,$\\
\multicolumn{4}{c}{${\mathcal L_{\theta}}(xy,z,t)+{\mathcal L_{\theta}}(xt,z,y)+{\mathcal L_{\theta}}(zy,x,t)+{\mathcal L_{\theta}}(zt,x,y)$} &$=$&$ 0.$
\end{longtable}

These elements will be called binary Leibniz {\it cocycles}. For a
linear map $f$ from $\bf A$ to  $\mathbb V$, if we define $\delta f\colon {\bf A} \times
{\bf A} \longrightarrow {\mathbb V}$ by $\delta f  (x,y ) =f(xy )$, then $\delta f\in {\rm Z}_{\rm BL}^{2}\left( {\bf A},{\mathbb V} \right) $. We define ${\rm B}^{2}\left({\bf A},{\mathbb V}\right) =\left\{ \theta =\delta f\ : f\in {\rm Hom}\left( {\bf A},{\mathbb V}\right) \right\} $.
We define the {\it second cohomology space} ${\rm H}_{\rm BL}^{2}\left( {\bf A},{\mathbb V}\right) $ as the quotient space ${\rm Z}_{\rm BL}^{2}
\left( {\bf A},{\mathbb V}\right) \big/{\rm B}^{2}\left( {\bf A},{\mathbb V}\right) $.

\item Let $({\bf A}, \cdot)$ be a mono  Leibniz algebra over  $\mathbb C$
and $\mathbb V$ be a vector space over ${\mathbb C}$. The $\mathbb C$-linear space ${\rm Z}_{\rm ML}^{2}\left(
\bf A,\mathbb V \right) $ is defined as the set of all  bilinear maps $\theta  \colon {\bf A} \times {\bf A} \longrightarrow {\mathbb V}$ such that
\begin{longtable}{lcl}
${\mathcal L_{\theta}}(x,y,z)+{\mathcal L_{\theta}}(y,x,z)+{\mathcal L_{\theta}}(y,z,x)+{\mathcal L_{\theta}}(x,z,y)+{\mathcal L_{\theta}}(z,x,y)+{\mathcal L_{\theta}}(z,y,x)$&$=$&$0,$\\
${\mathcal L_{\theta}}(x,y,zt)+{\mathcal L_{\theta}}(x,y,tz)+{\mathcal L_{\theta}}(x,z,yt)+{\mathcal L_{\theta}}(x,t,yz)+{\mathcal L_{\theta}}(x,z,ty)+{\mathcal L_{\theta}}(x,t,zy)$&$+$&\\
${\mathcal L_{\theta}}(y,x,zt)+{\mathcal L_{\theta}}(y,x,tz)+{\mathcal L_{\theta}}(z,x,yt)+{\mathcal L_{\theta}}(t,x,yz)+{\mathcal L_{\theta}}(z,x,ty)+{\mathcal L_{\theta}}(t,x,zy)$&$+$&\\
${\mathcal L_{\theta}}(y,z,xt)+{\mathcal L_{\theta}}(y,t,xz)+{\mathcal L_{\theta}}(z,y,xt)+{\mathcal L_{\theta}}(t,y,xz)+{\mathcal L_{\theta}}(z,t,xy)+{\mathcal L_{\theta}}(t,z,xy)$&$+$&\\
${\mathcal L_{\theta}}(y,z,tx)+{\mathcal L_{\theta}}(y,t,zx)+{\mathcal L_{\theta}}(z,y,tx)+{\mathcal L_{\theta}}(t,y,zx)+{\mathcal L_{\theta}}(z,t,yx)+{\mathcal L_{\theta}}(t,z,yx)$&$=$&$0.$
\end{longtable}

These elements will be called mono  Leibniz {\it cocycles}. We define the {\it second cohomology space} ${\rm H}_{\rm ML}^{2}\left( {\bf A},{\mathbb V}\right) $ as the quotient space ${\rm Z}_{\rm ML}^{2}
\left( {\bf A},{\mathbb V}\right) \big/{\rm B}^{2}\left( {\bf A},{\mathbb V}\right) $.
\end{itemize}
Now, when we say ${\bf A}$ algebra we mean binary or mono  Leibniz algebra.
Let $\operatorname{Aut}({\bf A}) $ be the automorphism group of  ${\bf A} $ and let $\phi \in \operatorname{Aut}({\bf A})$. For $\theta$ cocycles of the algebra ${\bf A}$, define  the action of the group $\operatorname{Aut}({\bf A}) $ on space of cocycles of the algebra ${\bf A}$ by $\phi \theta (x,y)=\theta \left( \phi \left( x\right) ,\phi \left( y\right) \right) $.  It is easy to verify that
 ${\rm B}^{2}\left( {\bf A},{\mathbb V}\right)$  is invariant under the action of $\operatorname{Aut}({\bf A}).$
 So, we have an induced action of  $\operatorname{Aut}({\bf A})$  on ${\rm H}_{\rm BL}^{2}\left( {\bf A},{\mathbb V}\right)$ (${\rm H}_{\rm ML}^{2}\left( {\bf A},{\mathbb V}\right)$).

Let $\bf A$ be a algebra of dimension $m$ over  $\mathbb C$ and ${\mathbb V}$ be a $\mathbb C$-vector
space of dimension $k$. For the bilinear map $\theta$, define on the linear space ${\bf A}_{\theta } = {\bf A}\oplus {\mathbb V}$ the
bilinear product `` $\left[ -,-\right] _{{\bf A}_{\theta }}$'' by $\left[ x+x^{\prime },y+y^{\prime }\right] _{{\bf A}_{\theta }}=
 xy +\theta(x,y) $ for all $x,y\in {\bf A},x^{\prime },y^{\prime }\in {\mathbb V}$.
The algebra ${\bf A}_{\theta }$ is called a $k$-{\it dimensional central extension} of ${\bf A}$ by ${\mathbb V}$. One can easily check that ${\bf A_{\theta}}$ is a binary (mono)
Leibniz algebra if and only if $\theta$ binary (mono)
Leibniz cocycle of the algebra ${\bf A}$.

Call the
set $\operatorname{Ann}(\theta)=\left\{ x\in {\bf A}:\theta \left( x, {\bf A} \right)+ \theta \left({\bf A} ,x\right) =0\right\} $
the {\it annihilator} of $\theta $. We recall that the {\it annihilator} of an  algebra ${\bf A}$ is defined as
the ideal $\operatorname{Ann}(  {\bf A} ) =\left\{ x\in {\bf A}:  x{\bf A}+ {\bf A}x =0\right\}$. Observe
 that
$\operatorname{Ann}\left( {\bf A}_{\theta }\right) =(\operatorname{Ann}(\theta) \cap\operatorname{Ann}({\bf A}))
 \oplus {\mathbb V}$.

A well-known result states that every algebra with a non-zero annihilator is a central extension of a smaller-dimensional algebra.

\begin{definition}
Let ${\bf A}$ be an algebra and $I$ be a subspace of $\operatorname{Ann}({\bf A})$. If ${\bf A}={\bf A}_0 \oplus I$
then $I$ is called an {\it annihilator component} of ${\bf A}$.
A central extension of an algebra $\bf A$ without annihilator component is called a {\it non-split central extension}.
\end{definition}

Our task is to find all central extensions of an algebra $\bf A$ by
a space ${\mathbb V}$.  In order to solve the isomorphism problem we need to study the
action of $\operatorname{Aut}({\bf A})$ on ${\rm H^{2}}\left( {\bf A},{\mathbb V}
\right) $. To do that, let us fix a basis $e_{1},\ldots ,e_{s}$ of ${\mathbb V}$, and $
\theta \in {\rm Z^{2}}\left( {\bf A},{\mathbb V}\right) $. Then $\theta $ can be uniquely
written as $\theta \left( x,y\right) =
\displaystyle \sum_{i=1}^{s} \theta _{i}\left( x,y\right) e_{i}$, where $\theta _{i}\in
{\rm Z^{2}}\left( {\bf A},\mathbb C\right) $. Moreover, $\operatorname{Ann}(\theta)=\operatorname{Ann}(\theta _{1})\cap\operatorname{Ann}(\theta _{2})\cap\ldots \cap\operatorname{Ann}(\theta _{s})$. Furthermore, $\theta \in
{\rm B^{2}}\left( {\bf A},{\mathbb V}\right) $\ if and only if all $\theta _{i}\in {\rm B^{2}}\left( {\bf A},
\mathbb C\right) $.
It is not difficult to prove (see \cite[Lemma 13]{hac16}) that given a  binary (resp. mono) Leibniz  algebra ${\bf A}_{\theta}$, if we write as
above $\theta \left( x,y\right) = \displaystyle \sum_{i=1}^{s} \theta_{i}\left( x,y\right) e_{i}\in {\rm Z^{2}}\left( {\bf A},{\mathbb V}\right) $ and
$\operatorname{Ann}(\theta)\cap \operatorname{Ann}\left( {\bf A}\right) =0$, then ${\bf A}_{\theta }$ has an
annihilator component if and only if $\left[ \theta _{1}\right] ,\left[
\theta _{2}\right] ,\ldots ,\left[ \theta _{s}\right] $ are linearly
dependent in ${\rm H^{2}}\left( {\bf A},\mathbb C\right) $.

\;

Let ${\mathbb V}$ be a finite-dimensional vector space over $\mathbb C$. The {\it Grassmannian} $G_{k}\left( {\mathbb V}\right) $ is the set of all $k$-dimensional
linear subspaces of $ {\mathbb V}$. Let $G_{s}\left( {\rm H^{2}}\left( {\bf A},\mathbb C\right) \right) $ be the Grassmannian of subspaces of dimension $s$ in
${\rm H^{2}}\left( {\bf A},\mathbb C\right) $. There is a natural action of $\operatorname{Aut}({\bf A})$ on $G_{s}\left( {\rm H^{2}}\left( {\bf A},\mathbb C\right) \right) $.
Let $\phi \in \operatorname{Aut}({\bf A})$. For $W=\left\langle
\left[ \theta _{1}\right] ,\left[ \theta _{2}\right] ,\dots,\left[ \theta _{s}
\right] \right\rangle \in G_{s}\left( {\rm H^{2}}\left( {\bf A},\mathbb C
\right) \right) $ define $\phi W=\left\langle \left[ \phi \theta _{1}\right]
,\left[ \phi \theta _{2}\right] ,\dots,\left[ \phi \theta _{s}\right]
\right\rangle $. We denote the orbit of $W\in G_{s}\left(
{\rm H^{2}}\left( {\bf A},\mathbb C\right) \right) $ under the action of $\operatorname{Aut}({\bf A})$ by $\operatorname{Orb}(W)$. Given
\[
W_{1}=\left\langle \left[ \theta _{1}\right] ,\left[ \theta _{2}\right] ,\dots,
\left[ \theta _{s}\right] \right\rangle ,W_{2}=\left\langle \left[ \vartheta
_{1}\right] ,\left[ \vartheta _{2}\right] ,\dots,\left[ \vartheta _{s}\right]
\right\rangle \in G_{s}\left( {\rm H^{2}}\left( {\bf A},\mathbb C\right)
\right),
\]
we easily have that if $W_{1}=W_{2}$, then $ \bigcap\limits_{i=1}^{s}\operatorname{Ann}(\theta _{i})\cap \operatorname{Ann}\left( {\bf A}\right) = \bigcap\limits_{i=1}^{s}
\operatorname{Ann}(\vartheta _{i})\cap\operatorname{Ann}( {\bf A}) $, and therefore we can introduce
the set
\[
{\bf T}_{s}({\bf A}) =\left\{ W=\left\langle \left[ \theta _{1}\right] ,
\left[ \theta _{2}\right] ,\dots,\left[ \theta _{s}\right] \right\rangle \in
G_{s}\left( {\rm H^{2}}\left( {\bf A},\mathbb C\right) \right) : \bigcap\limits_{i=1}^{s}\operatorname{Ann}(\theta _{i})\cap\operatorname{Ann}({\bf A}) =0\right\},
\]
which is stable under the action of $\operatorname{Aut}({\bf A})$.

\

Now, let ${\mathbb V}$ be an $s$-dimensional linear space and let us denote by
${\bf E}\left( {\bf A},{\mathbb V}\right) $ the set of all {\it non-split $s$-dimensional central extensions} of ${\bf A}$ by
${\mathbb V}$. By above, we can write
\[
{\bf E}\left( {\bf A},{\mathbb V}\right) =\left\{ {\bf A}_{\theta }:\theta \left( x,y\right) = \sum_{i=1}^{s}\theta _{i}\left( x,y\right) e_{i} \ \ \text{and} \ \ \left\langle \left[ \theta _{1}\right] ,\left[ \theta _{2}\right] ,\dots,
\left[ \theta _{s}\right] \right\rangle \in {\bf T}_{s}({\bf A}) \right\} .
\]
We also have the following result, which can be proved as in \cite[Lemma 17]{hac16}.

\begin{lemma}
 Let ${\bf A}_{\theta },{\bf A}_{\vartheta }\in {\bf E}\left( {\bf A},{\mathbb V}\right) $. Suppose that $\theta \left( x,y\right) =  \displaystyle \sum_{i=1}^{s}
\theta _{i}\left( x,y\right) e_{i}$ and $\vartheta \left( x,y\right) =
\displaystyle \sum_{i=1}^{s} \vartheta _{i}\left( x,y\right) e_{i}$.
Then the  binary (resp. mono) Leibniz algebras ${\bf A}_{\theta }$ and ${\bf A}_{\vartheta } $ are isomorphic
if and only if
$$\operatorname{Orb}\left\langle \left[ \theta _{1}\right] ,
\left[ \theta _{2}\right] ,\dots,\left[ \theta _{s}\right] \right\rangle =
\operatorname{Orb}\left\langle \left[ \vartheta _{1}\right] ,\left[ \vartheta
_{2}\right] ,\dots,\left[ \vartheta _{s}\right] \right\rangle .$$
\end{lemma}

This shows that there exists a one-to-one correspondence between the set of $\operatorname{Aut}({\bf A})$-orbits on ${\bf T}_{s}\left( {\bf A}\right) $ and the set of
isomorphism classes of ${\bf E}\left( {\bf A},{\mathbb V}\right) $. Consequently, we have a
procedure that allows us, given a  binary (resp. mono) Leibniz  algebra ${\bf A}'$ of
dimension $n-s$, to construct all non-split central extensions of ${\bf A}'$. This procedure is:

\begin{enumerate}
\item For a given binary (resp. mono)  Leibniz  algebra ${\bf A}'$ of dimension $n-s $, determine ${\rm H^{2}}( {\bf A}',\mathbb {C}) $, $\operatorname{Ann}({\bf A}')$ and $\operatorname{Aut}({\bf A}')$.

\item Determine the set of $\operatorname{Aut}({\bf A}')$-orbits on ${\bf T}_{s}({\bf A}') $.

\item For each orbit, construct the binary (resp. mono)  Leibniz  algebra associated with a
representative of it.
\end{enumerate}

\medskip

The above-described method gives all (Leibniz and non-Leibniz)  binary (resp. mono)  Leibniz  algebras. But we are interested in developing this method in such a way that it only gives non-Leibniz    binary (resp. mono)  Leibniz algebras because the classification of all five-dimensional nilpotent Leibniz algebras is given in \cite{leib5}. Clearly, any central extension of a non-Leibniz binary (resp. mono)   Leibniz  is non-Leibniz.
But a Leibniz  algebra may have extensions that are not binary (resp. mono)  Leibniz  algebras. More precisely, let ${\mathcal{L}}$ be a   Leibniz  algebra and
$\theta \in {\rm Z_{\rm BL}^2}({\mathcal{L}}, {\mathbb C})$  (resp. $\theta \in {\rm Z_{\rm ML}^2}({\mathcal{L}}, {\mathbb C})$).
Then ${\mathcal{L}}_{\theta }$ is a   Leibniz  algebra if and only if
\begin{equation*}
 \theta(xy,z)= \theta(xz,y)+\theta(x,yz)
 \end{equation*}
for all $x,y,z\in {\mathcal{L}}.$
Define the subspace
${\rm Z_{\mathcal{L}}^2}(\mathcal{L},{\mathbb C})$ of ${\rm Z_{\rm{BL}}^2}({\mathcal{L}},{\mathbb C})$ (resp. of ${\rm Z_{\rm{ML}}^2}({\mathcal{L}},{\mathbb C})$)
by
\begin{center}
${\rm Z_{\mathcal{L}}^2}({\mathcal{L}},{\mathbb C}) =\{  \theta \in
{\rm Z_{\rm BL}^2}({\mathcal{L}},{\mathbb C})$ (resp.  ${\rm Z_{\rm ML}^2}({\mathcal{L}},{\mathbb C})$)$:
\theta(xy,z)= \theta(xz,y)+\theta(x,yz),
\text{ for all } x, y,z\in {\mathcal{L}} \}.$
\end{center}

Observe that ${\rm B^2}({\mathcal{L}},{\mathbb C})\subseteq{\rm Z_{\rm{BL}}^2}({\mathcal{L}},{\mathbb C})$ (resp. ${\rm Z_{\rm{ML}}^2}({\mathcal{L}},{\mathbb C})$).
Let ${\rm H_{\mathcal{L}}^2}({\mathcal{L}},{\mathbb C}) =
{\rm Z_{\mathcal{L}}^2}({\mathcal{L}},{\mathbb C}) \big/{\rm B^2}({\mathcal{L}},{\mathbb C}).$
Then ${\rm H_{\mathcal{L}}^2}({\mathcal{L}},{\mathbb C})$ is a subspace of $
{\rm H_{\rm{BL}}^2}({\mathcal{L}},{\mathbb C}).$ Define
\begin{eqnarray*}
{\bf R}_{s}({\mathcal{L}})  &=&\left\{ {\mathcal W}\in {\bf T}_{s}({\mathcal{L}}) :{\mathcal W}\in G_{s}({\rm H_{\mathcal{L}}^2}({\mathcal{L}},{\mathbb C}) ) \right\}, \\
{\bf U}_{s}({\mathcal{L}})  &=&\left\{ {\mathcal W}\in {\bf T}_{s}({\mathcal{L}}) :{\mathcal W}\notin G_{s}({\rm H_{\mathcal{L}}^2}({\mathcal{L}},{\mathbb C}) ) \right\}.
\end{eqnarray*}
Then ${\bf T}_{s}({\mathcal{L}}) ={\bf R}_{s}(
{\mathcal{L}}) \mathbin{\mathaccent\cdot\cup} {\bf U}_{s}(
{\mathcal{L}}).$ The sets ${\bf R}_{s}({\mathcal{L}}) $
and ${\bf U}_{s}({\mathcal{L}})$ are stable under the action
of $\operatorname{Aut}({\mathcal{L}}).$
Thus, the  binary (resp. mono) Leibniz  algebras
corresponding to the representatives of $\operatorname{Aut}({\mathcal{L}}) $%
-orbits on ${\bf R}_{s}({\mathcal{L}})$ are   Leibniz  algebras,
while those corresponding to the representatives of $\operatorname{Aut}({\mathcal{L}}%
) $-orbits on ${\bf U}_{s}({\mathcal{L}})$ are non-Leibniz binary (resp. mono)  Leibniz algebras. Hence, we may construct all non-split non-Leibniz binary (resp. mono)  Leibniz algebras $%
\bf{A}$ of dimension $n$ with $s$-dimensional annihilator
from a given binary (resp. mono) Leibniz algebra $\bf{A}%
^{\prime }$ of dimension $n-s$ in the following way:

\begin{enumerate}
\item If $\bf{A}^{\prime }$ is non-Leibniz, then apply the procedure.

\item Otherwise, do the following:

\begin{enumerate}
\item Determine ${\bf U}_{s}(\bf{A}^{\prime })$ and $%
\operatorname{Aut}(\bf{A}^{\prime }).$

\item Determine the set of $\operatorname{Aut}(\bf{A}^{\prime })$-orbits on ${\bf U%
}_{s}(\bf{A}^{\prime }).$

\item For each orbit, construct the   binary (resp. mono)  Leibniz  algebra corresponding to one of its
representatives.
\end{enumerate}
\end{enumerate}

\subsubsection{Notations}
Let us introduce the following notations. Let ${\bf A}$ be a nilpotent algebra with
a basis $e_{1},e_{2}, \ldots, e_{n}.$ Then by $\Delta_{ij}$\ we will denote the
bilinear form
$\Delta_{ij}:{\bf A}\times {\bf A}\longrightarrow \mathbb C$
with $\Delta_{ij}(e_{l},e_{m}) = \delta_{il}\delta_{jm}.$
The set $\left\{ \Delta_{ij}:1\leq i, j\leq n\right\}$ is a basis for the linear space of
bilinear forms on ${\bf A},$ so every $\theta \in
{\rm Z}_{\rm BL}^2({\bf A},\bf \mathbb V )$ (resp. $\theta \in
{\rm Z}_{\rm ML}^2({\bf A},\bf \mathbb V )$ ) can be uniquely written as $
\theta = \displaystyle \sum_{1\leq i,j\leq n} c_{ij}\Delta _{{i}{j}}$, where $
c_{ij}\in \mathbb C$.
Let us fix the following notations for our nilpotent algebras:

$$\begin{array}{lll}

{\mathcal N}_{j}& \mbox{---}& j\mbox{th }3\mbox{-dimensional        $2$-step nilpotent algebra.} \\

{\mathfrak N}_{j}& \mbox{---}& j\mbox{th }4\mbox{-dimensional   $2$-step nilpotent algebra.} \\

\mathbb{M}_{j}& \mbox{---}& j\mbox{th }4\mbox{-dimensional   nilpotent  (non-Leibniz) mono  Leibniz algebra.} \\

{\bf B}_{j}& \mbox{---}& j\mbox{th }5\mbox{-dimensional    nilpotent  (non-Leibniz)  binary Leibniz algebra.} \\

\end{array}$$

\subsection{Classification of 5-dimensional nilpotent binary Leibniz algebras}

When we construct  algebras by   the central extension, the number of generators does not change.
So the generators of our algebra that we are looking at through the central extension must be equal to 3.
So we construct non-Leibniz binary Leibniz algebras through one-dimensional central extensions of 4 generated four dimensional nilpotent Leibniz algebras.
Such algebras are as follows.

{\small
\begin{longtable}{llllllll}
\hline

${\mathfrak N}_{01}$ & $:$ &  $e_1e_1 = e_2$ &&&&\\

${\rm H}_{\mathcal  L}^2({\mathfrak N}_{01})$ & $=$ & \multicolumn{6}{l}{$
\Big\langle
[\Delta_{13}],[\Delta_{14}],[\Delta_{21}],[\Delta_{31}],[\Delta _{33}],[\Delta _{34}],[\Delta _{41}],[\Delta _{43}],[\Delta _{44}]
\Big\rangle $}\\

${\rm H}_{\rm BL}^2({\mathfrak N}_{01})$ & $=$ & \multicolumn{6}{l}{${\rm H}_{\mathcal  L}^2({\mathfrak N}_{01})$}\\

\hline
{${\mathfrak N}_{03}$} &$:$ &  $e_1e_2=  e_3$ & $e_2e_1=-e_3$ &&& \\

${\rm H}_{\mathcal  L}^2({\mathfrak N}_{03})$ & $=$ & \multicolumn{6}{l}{$
\Big\langle
[\Delta_{11}],[\Delta_{12}],[\Delta_{14}],[\Delta_{22}],[\Delta_{24}],[\Delta_{13}-\Delta_{31}],[\Delta_{23}-\Delta_{32}],[\Delta _{41}],[\Delta _{42}],[\Delta_{44}]
\Big\rangle $}\\

${\rm H}_{\rm BL}^2({\mathfrak N}_{03})$ & $=$ & \multicolumn{6}{l}{${\rm H}_{\mathcal  L}^2({\mathfrak N}_{03})\oplus \langle[\Delta_{34}-\Delta_{43}] \rangle$}\\

\hline
${\mathfrak N}_{04}^{\alpha}$ &$:$ & $e_1e_1=  e_3$ & $e_1e_2=e_3$& $e_2e_2=\alpha e_3$  &&\\

${\rm H}_{\mathcal  L}^2({\mathfrak N}_{04}^{\alpha\neq0})$ & $=$ & \multicolumn{6}{l}{$
\Big\langle
[\Delta_{12}],[\Delta_{14}],[\Delta_{21}],[\Delta_{22}],[\Delta_{24}],
[\Delta _{41}],[\Delta _{42}],[\Delta_{44}]
\Big\rangle $}\\

${\rm H}_{\mathcal  L}^2({\mathfrak N}_{04}^{0})$ & $=$ & \multicolumn{6}{l}{$
\Big\langle
[\Delta_{12}],[\Delta_{14}],[\Delta_{21}],[\Delta_{22}],[\Delta_{24}],
[\Delta_{31}+\Delta_{32}],[\Delta _{41}],[\Delta _{42}],[\Delta_{44}]
\Big\rangle $}\\

${\rm H}_{\rm BL}^2({\mathfrak N}_{04}^{\alpha})$ & $=$ & \multicolumn{6}{l}{${\rm H}_{\mathcal  L}^2({\mathfrak N}_{04}^{\alpha})$}\\

\hline
${\mathfrak N}_{05}$ &$:$ & $e_1e_1=  e_3$& $e_1e_2=e_3$&  $e_2e_1=e_3$ &&\\

${\rm H}_{\mathcal  L}^2({\mathfrak N}_{05})$ & $=$ & \multicolumn{6}{l}{$
\Big\langle
[\Delta_{12}],[\Delta_{14}],[\Delta_{21}],[\Delta_{22}],[\Delta_{24}],
[\Delta _{41}],[\Delta _{42}],[\Delta_{44}]
\Big\rangle $}\\

${\rm H}_{\rm BL}^2({\mathfrak N}_{05})$ & $=$ & \multicolumn{6}{l}{${\rm H}_{\mathcal  L}^2({\mathfrak N}_{05})$}\\

\hline
${\mathfrak N}_{06}$ &$:$ & $e_1e_2 = e_4$& $e_3e_1 = e_4$   &&&\\

${\rm H}_{\mathcal  L}^2({\mathfrak N}_{06})$ & $=$ & \multicolumn{6}{l}{ $
\Big\langle
[\Delta_{11}],[\Delta_{13}],[\Delta_{21}],[\Delta_{22}],
[\Delta _ {23}],[\Delta_{31}], [\Delta _{32}],[\Delta_{33}]
\Big\rangle $}\\

${\rm H}_{\rm BL}^2({\mathfrak N}_{06})$ & $=$ & \multicolumn{6}{l}{${\rm H}_{\mathcal  L}^2({\mathfrak N}_{06})$}\\

\hline
${\mathfrak N}_{09}^{\alpha}$ &$:$ & $e_1e_1 = e_4$ & $e_1e_2 = \alpha e_4$ &  $e_2e_1 = -\alpha e_4$ & $e_2e_2 = e_4$ &  $e_3e_3 = e_4$\\

${\rm H}_L^2({\mathfrak N}_{09}^{\alpha})$ & $=$ & \multicolumn{6}{l}{$
\Big\langle
[\Delta_{12}],[\Delta_{13}],[\Delta_{21}],[\Delta_{22}],[\Delta_{23}],
[\Delta_{31}],[\Delta_{32}],[\Delta_{33}]
\Big\rangle $}\\

${\rm H}_{\rm BL}^2({\mathfrak N}_{09}^{\alpha})$ & $=$ & \multicolumn{6}{l}{${\rm H}_{\mathcal  L}^2({\mathfrak N}_{09}^{\alpha})$}\\

\hline
${\mathfrak N}_{10}$ &$:$ &  $e_1e_2 = e_4$ & $e_1e_3 = e_4$ & $e_2e_1 = -e_4$ & $e_2e_2 = e_4$ & $e_3e_1 = e_4$  \\

${\rm H}_L^2({\mathfrak N}_{10})$ & $=$ & \multicolumn{6}{l}{$
\Big\langle
[\Delta_{11}],[\Delta_{13}],[\Delta_{21}],[\Delta_{22}],[\Delta_{23}],
[\Delta_{31}],[\Delta_{32}],[\Delta_{33}]
\Big\rangle $}\\

${\rm H}_{\rm BL}^2({\mathfrak N}_{10})$ & $=$ & \multicolumn{6}{l}{${\rm H}_{\mathcal  L}^2({\mathfrak N}_{10})$}\\

\hline
${\mathfrak N}_{11}$ &$:$ &  $e_1e_1 = e_4$ & $e_1e_2 = e_4$ & $e_2e_1 = -e_4$ & $e_3e_3 = e_4$&  \\

${\rm H}_L^2({\mathfrak N}_{11})$ & $=$ & \multicolumn{6}{l}{$
\Big\langle
[\Delta_{12}],[\Delta_{13}],[\Delta_{21}],[\Delta_{22}],[\Delta_{23}],
[\Delta_{31}],[\Delta_{32}],[\Delta_{33}]
\Big\rangle $}\\

${\rm H}_{\rm BL}^2({\mathfrak N}_{11})$ & $=$ & \multicolumn{6}{l}{${\rm H}_{\mathcal  L}^2({\mathfrak N}_{11})$}\\

\hline
${\mathfrak N}_{15}$ &$:$ &  $e_1e_2 = e_4$ & $e_2e_1 = -e_4$ & $e_3e_3 = e_4$ && \\

${\rm H}_L^2({\mathfrak N}_{15})$ & $=$ & \multicolumn{6}{l}{$
\Big\langle
[\Delta_{11}],[\Delta_{13}],[\Delta_{21}],[\Delta_{22}],[\Delta_{23}],
[\Delta_{31}],[\Delta_{32}],[\Delta_{33}]
\Big\rangle $}\\

${\rm H}_{\rm BL}^2({\mathfrak N}_{15})$ & $=$ & \multicolumn{6}{l}{
${\rm H}_{\mathcal  L}^2({\mathfrak N}_{15})$}\\

\hline
\end{longtable}}

From the previous table, we obtain that only algebra $\mathfrak N_{03}$ has a non-Leibniz binary Leibniz central extension.

\subsubsection{Central extensions of ${\mathfrak N}_{03}$}
Let us use the following notations:

\begin{longtable}{lllllll}
$\nabla_1 = [\Delta_{11}],$ & $\nabla_2 = [\Delta_{12}],$ & $\nabla_3 = [\Delta_{13}-\Delta_{31}],$ &
$\nabla_4 = [\Delta_{14}],$ \\
$\nabla_5 = [\Delta_{22}],$ & $\nabla_6 = [\Delta_{23}-\Delta_{32}],$ &  $\nabla_7 = [\Delta_{24}],$ &$\nabla_8 = [\Delta_{41}],$ \\ $\nabla_9 = [\Delta_{42}],$ &
$\nabla_{10} = [\Delta_{44}],$ &$\nabla_{11} = [\Delta_{34}-\Delta_{43}].$\\
\end{longtable}	

Take $\theta=\sum\limits_{i=1}^{11}\alpha_i\nabla_i\in {\rm H}_{\rm BL}^2(\mathfrak N_{03}).$
	The automorphism group of $\mathfrak N_{03}$ consists of invertible matrices of the form
	$$
	\phi=
	\left(
\begin{array}{cccc}
 x & z & 0 & 0 \\
 y & t & 0 & 0 \\
 u & q & xt-yz & w \\
 v & p & 0 & r
\end{array}
\right).
	$$

	Since
	$$
	\phi^T\begin{pmatrix}
	\alpha_1 &  \alpha_2 & \alpha_3 & \alpha_4\\
	0        & \alpha_5  & \alpha_6 & \alpha_7\\
	-\alpha_3& -\alpha_6 & 0        & \alpha_{11}\\
	\alpha_8& \alpha_9   & -\alpha_{11} & \alpha_{10}\\
	\end{pmatrix} \phi=
	\begin{pmatrix}
	\alpha_1^*  &  -\alpha^*+\alpha_2^* & \alpha_3^* & \alpha_4^*\\
	\alpha^*           & \alpha_5^*  & \alpha_6^* & \alpha_7^*\\
	-\alpha_3^* & -\alpha_6^* & 0        & \alpha_{11}^*\\
	\alpha_8^*  & \alpha_9^*   & -\alpha_{11}^* & \alpha_{10}^*\\
	\end{pmatrix},
	$$
 we have that the action of ${\rm Aut} (\mathfrak N_{03})$  on the subspace
$\langle \sum\limits_{i=1}^{11}\alpha_i\nabla_i  \rangle$
is given by $\langle \sum\limits_{i=1}^{11}\alpha_i^{*}\nabla_i\rangle,$
where

\begin{longtable}{lcl}
$\alpha^*_1$&$=$&$x^2 \alpha _1+x y \alpha _2+v x \alpha _4+y^2 \alpha _5+v y \alpha _7+v x \alpha _8+v y \alpha _9+v^2 \alpha _{10},$ \\
$\alpha^*_2$&$=$&$2 x z \alpha _1+(t x+y z)\alpha_2+(px+vz)\alpha _4+2ty\alpha _5+$\\
&&$(tv+py)\alpha _7+(p x+vz)\alpha _8+(tv+py)\alpha _9+2pv\alpha _{10}$\\
$\alpha^*_3$&$=$&$(x t-y z)(x \alpha_3+y \alpha_6-v\alpha_{11}),$ \\
$\alpha_4^*$&$=$&$r(x\alpha_4+y \alpha_7+v \alpha_{10}+u \alpha_{11})+w(x\alpha_3+y \alpha _6-v \alpha_{11}),$\\
$\alpha_5^*$&$=$&$z^2 \alpha _1+t z \alpha _2+p z \alpha _4+t^2 \alpha _5+p t \alpha _7+p z \alpha _8+p t \alpha _9+p^2 \alpha _{10},$\\
$\alpha_6^*$&$=$&$(t x-y z)(z \alpha _3+t \alpha _6-p \alpha _{11}),$\\
$\alpha_7^*$&$=$&$w(z\alpha_3+t\alpha_6-p\alpha_{11})+r(z\alpha_4+t\alpha_7+p\alpha_{10}+q\alpha_{11}),$\\
$\alpha_8^*$&$=$&$-w x \alpha _3-w y \alpha _6+r x \alpha _8+r y \alpha _9+r v \alpha _{10}-r u \alpha_{11}+v w\alpha _{11},$\\
$\alpha_9^*$&$=$&$-w z \alpha _3-t w \alpha _6+r z \alpha _8+r t \alpha _9+p r \alpha _{10}-q r \alpha _{11}+p w \alpha _{11},$\\
$\alpha_{10}^*$&$=$&$r^2 \alpha _{10},$\\
$\alpha_{11}^*$&$=$&$r (t x-y z) \alpha _{11}.$\\
\end{longtable}

Since ${\rm H}_{\rm BL}^2(\mathfrak N_{03})={\rm H}_{\mathcal  L}^2(\mathfrak N_{03})\oplus \langle [\Delta_{34}-\Delta_{43}]\rangle$ and we are interested only in new algebras, we have  $\alpha_{11}\neq0.$

Then putting
$$\begin{array}{lcllcl}
v&=&(x \alpha_3+y \alpha_6)\alpha_{11}^{-1}, &
u&=&-((x \alpha _3 +y \alpha _6) \alpha _{10}+(x \alpha _4+y \alpha _7) \alpha _{11}) \alpha_{11}^{-2},\\
p&=&(z\alpha_3+t\alpha_6)\alpha_{11}^{-1}, &
q&=&-((z \alpha _3+t \alpha _6) \alpha _{10}+\left(z \alpha _4+t \alpha _7\right) \alpha _{11})\alpha _{11}^{-2},
\end{array}$$
we have
$$\alpha^*_3=\alpha^*_4=\alpha_6^*\alpha_7^*=0.$$

Thus, without loss of generality, we can suppose  $ \alpha_3 = \alpha_4 = \alpha_6 = \alpha_7 = 0 $.
Then $ p = v = u = q = 0 $  and we have the following relations:
\begin{longtable}{lll}
$\alpha^*_1$ &$=$&$x^2 \alpha _1+x y \alpha _2+y^2 \alpha _5,$\\
$\alpha_2^*$&$=$&$2 x z \alpha _1+(t x+y z) \alpha _2+2 t y \alpha _5,$\\
$\alpha_5^*$&$=$&$z^2 \alpha _1+t z \alpha _2+t^2 \alpha _5,$\\
$\alpha_8^*$&$=$&$r \left(x \alpha _8+y \alpha _9\right),$\\
$\alpha_9^*$&$=$&$r \left(z \alpha _8+t \alpha _9\right),$\\
$\alpha_{10}^*$&$=$&$r^2 \alpha _{10},$\\
$\alpha_{11}^*$&$=$&$r (t x-y z) \alpha_{11}.$
\end{longtable}

\begin{enumerate}
    \item  Let $\alpha_{1}=\alpha_{2}=\alpha_{5}=\alpha_{8}=\alpha_{9}=0$. Then we have the following subcases:
\begin{enumerate}
\item if $\alpha _{10}=0,$ then choosing $r=\frac{1}{(t x-y z) \alpha _{11}}$, we have the representative $\langle \nabla_{11} \rangle$;
\item if $\alpha _{10} \neq 0,$ then choosing $r=\frac{(t x-y z) \alpha _{11}}{\alpha_{10}},$
    we have the representative $\langle \nabla_{10}+\nabla_{11} \rangle$.

\end{enumerate}

\item  Let $\alpha_{1}=\alpha_{2}=\alpha_{5}=0, (\alpha_{8},\alpha_{9})\neq (0,0)$. Then without loss of generality, we can assume $\alpha_8\neq 0$ and consider the following subcases:
\begin{enumerate}
\item if $\alpha _{10}=0,$ then choosing
$z=-\frac{\alpha_9}{\alpha_{11}}$ and $t=\frac{\alpha_8}{\alpha_{11}}$, we have the representative $\langle \nabla_{8}+\nabla_{11} \rangle$;

\item if $\alpha _{10} \neq 0,$ then choosing
$x=0,$ $y=1,$
$r=\frac{\alpha _9}{\alpha_{10}},$
$z=-\frac{\alpha_9}{\alpha_{11}}$ and
$   t=\frac{\alpha_8}{\alpha_{11}},$
    we have the representative $\langle \nabla_{8}+\nabla_{10}+\nabla_{11} \rangle$.
\end{enumerate}

\item Let $(\alpha_{1},\alpha_{5})=(0,0).$ Then $\alpha_2\neq0$ and taking $ y = 0, \ z = 0 $ and we get that:
\begin{longtable}{lcllcllcllcl}
$\alpha^*_1$&$=$&$0,$& $\alpha_2^*$&$=$&$tx\alpha _2,$ & $\alpha_5^*$&$=$&$0,$ &
$\alpha_8^*$&$=$&$rx\alpha _8,$\\
$\alpha_9^*$&$=$&$rt\alpha _9,$ &
$\alpha_{10}^*$&$=$&$r^2 \alpha _{10},$ & $\alpha_{11}^*$&$=$&$rtx\alpha_{11}.$
\end{longtable}

Then we have the following cases:
\begin{enumerate}
\item $\alpha _{8}=\alpha _{9}=\alpha _{10}=0,$ then choosing $r=\frac{\alpha_2}{\alpha_{11}},$ we have the representative $\langle \nabla_{2}+\nabla_{11} \rangle$;

\item $\alpha _{8}=\alpha _{9}=0, \alpha _{10}\neq 0,$ then choosing $t=1,$ $r=\frac{\alpha_2}{\alpha_{11}}$ and $ x=\frac{\alpha_2\alpha_{10}}{\alpha_{11}^2},$ we have the representative  $\langle \nabla_{2}+\nabla_{10}+\nabla_{11} \rangle$;

\item $\alpha _{8}\neq 0, \alpha_9=0,  \alpha _{10}=0,$ then choosing $r=\frac{\alpha_2}{\alpha_{11}}$ and $t=\frac{\alpha_8}{\alpha_{11}},$ we have the representative  $\langle \nabla_{2}+\nabla_{8}+\nabla_{11} \rangle$;

\item $\alpha_{8}=0, \alpha_{9}\neq 0,  \alpha_{10}=0,$ in this case choosing the suitable automorphism we can obtain
$\alpha _ {8} \neq 0, \alpha_9 = 0,  \alpha_ {10} = 0,$ which is the case considered above;

\item $\alpha _{8}\neq 0, \alpha_9=0, \alpha _{10}\neq 0,$ then choosing
$t=\frac{\alpha_8}{\alpha_{11}},$ $r=\frac{\alpha_2}{\alpha_{11}}$ and $ x=\frac{\alpha_2\alpha_{10}}{\alpha_{8}\alpha_{11}},$ we have the representative $\langle \nabla_{2}+\nabla_{8}+\nabla_{10}+\nabla_{11} \rangle$;

\item $\alpha _{8}=0, \alpha_9\neq 0, \alpha _{10}\neq 0,$ in this case choosing the suitable automorphism we can obtain  $\alpha _{8}\neq 0, \alpha_9=0, \alpha _{10}\neq 0,$ which is the case considered above;

\item $\alpha _{8}\neq 0, \alpha _{9}\neq0,$ then choosing $r=\frac{\alpha_2}{\alpha_{11}},$  $t=\frac{\alpha_8}{\alpha_{11}}$ and $x=\frac{\alpha_9}{\alpha_{11}},$ we have the family of representatives $\langle \nabla_{2}+\nabla_{8}+\nabla_{9}+\alpha\nabla_{10}+\nabla_{11} \rangle$.
\end{enumerate}

\item Let $(\alpha_{1},\alpha_{5})\neq (0,0)$. Then without loss of generality, we can assume $\alpha_1\neq 0.$
If $\alpha_2^2\neq4\alpha_1\alpha_5$, then choosing
\begin{center}
$x=-\alpha_2+\sqrt{\alpha_2^2-4\alpha_1\alpha_5},\quad z=-\alpha_2-\sqrt{\alpha_2^2-4\alpha_1\alpha_5},\quad
y= 2 \alpha_1,\quad t=2\alpha_1,$
\end{center}
we have $\alpha_1^*=0$ and $\alpha_5^*=0$ and it is the case considered above.

Therefore, we can suppose $(\alpha_{1},\alpha_{5})\neq (0,0)$ and $\alpha_2^2=4 \alpha_1 \alpha_5$.
By taking $ z=-\frac{t \alpha_2} {2 \alpha_1},$ we have the following \begin{longtable}{lcllcllcl}
$\alpha^*_1$&$=$&$\frac{\left(2 x \alpha _1+y \alpha _2\right){}^2}{4 \alpha _1},$&
$\alpha_2^*$&$=$&$0,$ &
$\alpha_5^*$&$=$&$0,$ \\
$\alpha_8^*$&$=$&$r \left(x \alpha _8+y \alpha _9\right),$ &
$\alpha_9^*$&$=$&$\frac{rt\left( 2\alpha_1\alpha _9- \alpha _2 \alpha _8\right)}{2\alpha_1},$\\
$\alpha_{10}^*$&$=$&$r^2 \alpha _{10},$&
$\alpha_{11}^*$&$=$&$r \left(t x+\frac{t y \alpha _2}{2 \alpha _1}\right) \alpha _{11}.$
\end{longtable}

Without loss of generality, one can assume $\alpha_2 = 0$ obtain
\begin{longtable}{lcllcllcl}
$\alpha^*_1$&$=$&$x^2\alpha _1,$& $\alpha_8^*$&$=$&$r\left(x\alpha_8+y\alpha_9\right),$&
$\alpha_9^*$&$=$&$rt\alpha_9,$\\
$\alpha_{10}^*$&$=$&$r^2 \alpha _{10},$&
$\alpha_{11}^*$&$=$&$rtx\alpha _{11}.$
\end{longtable}

Now consider following cases:
\begin{enumerate}
  \item $\alpha _9=0,\alpha_8=0, \alpha_{10}=0,$ then choosing
  $z=0,$
  $x=1,$ $r=1$ and $t=\frac{\alpha _1}{\alpha_{11}},$ we have the representative $\langle \nabla_{1}+\nabla_{11} \rangle$;

  \item $\alpha _9=0,\alpha_8=0, \alpha_{10}\neq 0,$ then choosing $z=0,$ $x=\sqrt{\alpha_{10}},$
  $t=\frac{\sqrt{\alpha _1\alpha_{10}}}{\alpha_{11}}$
  and
  $r=\sqrt{\alpha_{1}},$ we have the representative $\langle \nabla_{1}+\nabla_{10}+\nabla_{11} \rangle$;

 % \item $\alpha _9=0,\alpha_8\neq0, \alpha_{10}=0$ then choosing
 % $z=0,$ $x=1,$ $t=\frac{\alpha_8}{\alpha_{11}}$ and
 % $r=\frac{\alpha _1}{\alpha_{8}}$ we have the representative $\langle \nabla_{1}+\nabla_{8}+\nabla_{11} \rangle$.

  \item $\alpha _9=0,\alpha_8\neq0,$ then choosing
  $z=0,$
  $x=1,$
  $t=\frac{\alpha_8}{\alpha_{11}}$ and
  $r=\frac{\alpha_1}{\alpha_{8}},$
  we have the family of representatives $\langle \nabla_{1}+\nabla_{8}+\alpha\nabla_{10}+\nabla_{11} \rangle$;

  \item $\alpha _9\neq 0,\alpha_{10}=0,$ then choosing
  $z=0,$
  $x=\frac{\alpha_9}{\alpha_{11}},$
  $r=1,$
  $y=-\frac{\alpha_8}{\alpha_{11}}$ and
  $t=\frac{\alpha _1\alpha _9}{\alpha_{11}^2},$ we have the representative $\langle \nabla_{1}+\nabla_{9}+\nabla_{11} \rangle$;

  \item $\alpha _9\neq 0,\alpha_{10}\neq 0,$ then choosing
  $z=0,$
  $x=\frac{\alpha_9}{\alpha_{11}},$
  $y=-\frac{\alpha_8}{\alpha_{11}},$ $t=\frac{\sqrt{\alpha_1\alpha_{10}}}{\alpha_{11}},$ and $r=\frac{\alpha_9\sqrt{\alpha_1}}{\alpha_{11}\sqrt\alpha_{10}},$ we have $\langle \nabla_{1}+\nabla_{9}+\nabla_{10}+\nabla_{11} \rangle$.
\end{enumerate}

\end{enumerate}

Summarizing, we have the following distinct orbits
\begin{center}
$\langle \nabla_{11}\rangle,$
$\langle \nabla_{10}+\nabla_{11} \rangle,$
$\langle \nabla_{8}+\nabla_{11}\rangle,$
$\langle \nabla_{8}+\nabla_{10}+\nabla_{11} \rangle,$
$\langle \nabla_{2}+\nabla_{11}\rangle,$
$\langle \nabla_{2}+\nabla_{8}+\nabla_{11} \rangle,$
$\langle \nabla_{2}+\nabla_{10}+\nabla_{11} \rangle,$
$\langle \nabla_{2}+\nabla_{8}+\nabla_{10}+\nabla_{11} \rangle,$
$\langle \nabla_{2}+\nabla_{8}+\nabla_{9}+\alpha\nabla_{10}+\nabla_{11} \rangle,$
$\langle \nabla_{1}+\nabla_{11}\rangle,$
$\langle \nabla_{1}+\nabla_{10}+\nabla_{11} \rangle,$
$\langle \nabla_{1}+\nabla_{8}+\alpha \nabla_{10}+\nabla_{11} \rangle,$
$\langle \nabla_{1}+\nabla_{9}+\nabla_{11} \rangle,$
$\langle \nabla_{1}+\nabla_{9}+\nabla_{10}+\nabla_{11} \rangle.$
\end{center}

\subsubsection{The classification theorem}\label{secteoA}

Now we are ready to summarize all results related to the algebraic classification of complex $5$-dimensional nilpotent binary Leibniz algebras.

\begin{theoremA}
Let ${\bf B}$ be a complex  $5$-dimensional nilpotent binary Leibniz algebra.
Then ${\bf B}$ is a Leibniz  algebra or isomorphic to one algebra from the following list:

\begin{longtable}{llllllllll}
${\bf B}_{01}$ & $:$ &
$e_1 e_2 = e_3$&  $e_2 e_1=-e_3$& $e_3 e_4=e_5$& $e_4e_3=-e_5$ \\
${\bf B}_{02}$ &$:$ &
$e_1 e_2 = e_3$&  $e_2 e_1=-e_3$& $e_3 e_4=e_5$& $e_4e_3=-e_5$&$e_4e_4=e_5$\\
${\bf B}_{03}$ &$:$&
$e_1 e_2 = e_3$&  $e_2 e_1=-e_3$& $e_3 e_4=e_5$& $e_4e_1=e_5$&$e_4e_3=-e_5$ \\
${\bf B}_{04}$ &$:$ &
$e_1 e_2 = e_3$ &  $e_2 e_1=-e_3$& $e_3 e_4=e_5$ & $e_4e_1=e_5$&$e_4e_3=-e_5$&$e_4e_4=e_5$\\
${\bf B}_{05}$ &$:$ &
$e_1 e_2 = e_3+e_5$&  $e_2 e_1=-e_3$& $e_3 e_4=e_5$&$e_4e_3=-e_5$ \\
${\bf B}_{06}$ &$:$ &
$e_1 e_2 = e_3+e_5$&  $e_2 e_1=-e_3$ & $e_3 e_4=e_5$& $e_4e_1=e_5$&$e_4e_3=-e_5$\\
${\bf B}_{07}$ &$:$ &
$e_1 e_2 = e_3+e_5$&  $e_2 e_1=-e_3$& $e_3 e_4=e_5$&$e_4e_3=-e_5$& $e_4e_4=e_5$\\
${\bf B}_{08}$ &$:$&
$e_1 e_2 = e_3+e_5$&  $e_2 e_1=-e_3$& $e_3 e_4=e_5$& $e_4e_1=e_5$&$e_4e_3=-e_5$& $e_4e_4=e_5$\\
${\bf B}_{09}^\alpha$ &$:$&
$e_1 e_2 = e_3+e_5$&  $e_2 e_1=-e_3$& $e_3 e_4=e_5$& $e_4e_1=e_5$&\\
&&$e_4e_2=e_5$&$e_4e_3=-e_5$& $e_4e_4=\alpha e_5$\\
${\bf B}_{10}$ &$:$ &
$e_1 e_1 = e_5$&$e_1 e_2 = e_3$&  $e_2 e_1=-e_3$& $e_3 e_4=e_5$&$e_4e_3=-e_5$ \\
${\bf B}_{11}$ &$:$ &
$e_1 e_1 = e_5$&$e_1 e_2 = e_3$&  $e_2 e_1=-e_3$& $e_3 e_4=e_5$ &$e_4e_3=-e_5$& $e_4e_4=e_5$\\
%${\bf B}_{12}$ &$:$ &
%$e_1 e_1 = e_5$&$e_1 e_2 = e_3$&  $e_2 e_1=-e_3$& $e_3 e_4=e_5$& %$e_4e_1=e_5$&$e_4e_3=-e_5$\\
${\bf B}_{12}^\alpha$ &$:$&
$e_1 e_1 = e_5$&$e_1 e_2 = e_3$&  $e_2 e_1=-e_3$& $e_3 e_4=e_5$&\\ &&$e_4e_1=e_5$&$e_4e_3=-e_5$& $e_4e_4= \alpha e_5$\\
${\bf B}_{13}$ &$:$ &
$e_1 e_1 = e_5$&$e_1 e_2 = e_3$&  $e_2 e_1=-e_3$& $e_3 e_4=e_5$&$e_4 e_2=e_5$&$e_4e_3=-e_5$ \\
${\bf B}_{14}$ &$:$&
$e_1 e_1 = e_5$&$e_1 e_2 = e_3$&  $e_2 e_1=-e_3$& $e_3 e_4=e_5$\\&&
$e_4 e_2=e_5$&$e_4e_3=-e_5$& $e_4e_4=e_5$
\end{longtable}
\end{theoremA}

\subsection{Classification of $4$-dimensional nilpotent mono  Leibniz algebras }
For   algebras  constructed by the method of central extension, the number of generators do not change.
So the generators of our algebra that we are looking  for must be equal to $2$. So we construct nilpotent non-Leibniz mono  Leibniz algebras through one-dimensional central extensions of
$3$-dimensional $2$-generated nilpotent Leibniz algebras.
Such algebras are as follows.

\begin{longtable}{ll llllll}
\hline
\multicolumn{8}{c}{{\bf The list of 2-step nilpotent 3-dimensional mono Leibniz algebras}}  \\
\hline

${\mathcal N}_{01}$ & $:$&$e_1e_1 = e_2$ &&&&\\
${\rm H}_{\rm BL}^2({\mathcal N}_{01})$ & $=$ & \multicolumn{6}{l}{$
\Big\langle
[\Delta_{13}],[\Delta_{21}],[\Delta_{31}],[\Delta_{33}]
\Big\rangle $}\\

 ${\rm H}_{\rm ML}^2({\mathcal N}_{01})$ & $=$ & \multicolumn{6}{l}{${\rm H}_{\rm BL}^2({\mathcal N}_{01})\oplus \langle [\Delta _{23}]\rangle$}\\

%${\rm H}_{\mathfrak n^3}^2({\mathcal N}_{01})=\Big\langle
%[\Delta_{13}],[\Delta_{31}],[\Delta_{33}]
%\Big\rangle $\\
\hline
${\mathcal N}_{02}$ & $:$ & $e_1e_2=e_3$ & $e_2e_1=-e_3$ &&& \\

${\rm H}_{\rm BL}^2({\mathcal N}_{02})$ & $=$ & \multicolumn{6}{l}{$
\Big\langle
[\Delta _{11}],[\Delta _{12}],[\Delta_{13}-\Delta _{31}],[\Delta_{22}],[\Delta _{23}-\Delta_{32}]
\Big\rangle $}\\

${\rm H}_{\rm ML}^2({\mathcal N}_{02})$ & $=$ & \multicolumn{6}{l}{${\rm H}_{\rm BL}^2({\mathcal N}_{02})\oplus \langle [\Delta_{31}], [\Delta_{32}], [\Delta_{33}]\rangle$}\\

%${\rm H}_{\mathfrak n^3}^2({\mathcal N}_{02})={\rm H}_{\rm ML}^2({\mathcal N}_{02})$\\

\hline
${\mathcal N}_{03}^{\alpha}$ & $:$ & $e_1e_1= e_3$ & $e_1e_2=e_3$ & $e_2e_2=\alpha e_3$  &&\\

${\rm H}_{\rm BL}^2({\mathcal N}_{03}^{\alpha\neq0})$ & $=$ & \multicolumn{6}{l}{$
\Big\langle
[\Delta_{12}],[\Delta_{21}],[\Delta_{22}]
\Big\rangle $}\\

${\rm H}_{\rm ML}^2({\mathcal N}_{03}^{\alpha\neq0})$ & $=$ & \multicolumn{6}{l}{${\rm H}_{\rm BL}^2({\mathcal N}_{03}^{\alpha\neq0})\oplus \langle [\Delta_{31}+\Delta_{32}], [\Delta_{32}]\rangle$}\\

${\rm H}_{\rm BL}^2({\mathcal N}_{03}^{0})$ & $=$ & \multicolumn{6}{l}{$
\Big\langle
[\Delta_{12}],[\Delta_{21}],[\Delta_{22}], [\Delta_{31}+\Delta_{32}]
\Big\rangle $}\\

${\rm H}_{\rm ML}^2({\mathcal N}_{03}^{0})$ & $=$ & \multicolumn{6}{l}{${\rm H}_{\rm BL}^2({\mathcal N}_{03}^{0})\oplus \langle [\Delta_{32}]\rangle$}\\

%${\rm H}_{\mathfrak n^3}^2({\mathcal N}_{03}^{\alpha})=\Big\langle
%[\Delta_{12}],[\Delta_{21}],[\Delta_{22}]
%\Big\rangle $\\

\hline
${\mathcal N}_{04}$ & $:$ & $e_1e_1= e_3$ & $e_2e_2=e_3$ &&\\

${\rm H}_{\rm BL}^2({\mathcal N}_{04})$ & $=$ & \multicolumn{6}{l}{
$\Big\langle
[\Delta_{12}],[\Delta_{21}],[\Delta_{22}]
\Big\rangle $}\\

${\rm H}_{\rm ML}^2({\mathcal N}_{04})$ & $=$ & \multicolumn{6}{l}{${\rm H}_{\rm BL}^2({\mathcal N}_{04})\oplus \langle [\Delta_{31}], [\Delta_{32}]\rangle$}\\

%${\rm H}_{\mathfrak n^3}^2({\mathcal N}_{04})=\Big\langle
%[\Delta_{12}],[\Delta_{21}],[\Delta_{22}]
%\Big\rangle $\\

\hline
\end{longtable}

%Let ${\mathfrak n^3}$ be the variety of nilpotent algebras of nil-index $3.$
%In the present table, the cohomology spaces ${\rm H}_{\mathfrak n^3}^2$
%were defined by a similar way with ${\rm H}_{\rm ML}^2.$

\subsubsection{Central extensions of ${\mathcal N}_{01}$}
	Let us use the following notations:
	\begin{longtable}{lllllll}
	$\nabla_1 = [\Delta_{13}],$ & $\nabla_2 = [\Delta_{21}],$ &$\nabla_3 = [\Delta_{31}],$ & $\nabla_4 = [\Delta_{23}],$ & $\nabla_5 = [\Delta_{33}].$
	\end{longtable}	
	
Take $\theta=\sum\limits_{i=1}^5\alpha_i\nabla_i\in {\rm H_{\rm ML}^2}({\mathcal N}_{01}).$
	The automorphism group of ${\mathcal N}_{01}$ consists of invertible matrices of the form
	$$\phi=
	\begin{pmatrix}
	x &  0  & 0\\
	y &  x^2& z\\
	u &  0  & t\\
    \end{pmatrix}.$$

 Since
	$$
	\phi^T\begin{pmatrix}
	0 & 0  & \alpha_1\\
	\alpha_2  & 0 & \alpha_4\\
	\alpha_3&  0    & \alpha_5\\
	\end{pmatrix} \phi=	\begin{pmatrix}
	\alpha^* & 0  & \alpha_1^*\\
	\alpha_2^*& 0 & \alpha_4^*\\
	\alpha_3^*&  0    & \alpha_5^*\\
	\end{pmatrix},
	$$
	 we have that the action of ${\rm Aut} ({\mathcal N}_{01})$ on the subspace
$\langle \sum\limits_{i=1}^5\alpha_i\nabla_i  \rangle$
is given by
$\langle \sum\limits_{i=1}^5\alpha_i^{*}\nabla_i\rangle,$
where
\begin{longtable}{lcl}
$\alpha^*_1$&$=$& $t(x\alpha_1+y\alpha_4+u \alpha_5),$ \\
$\alpha^*_2$&$=$&$x^2(x\alpha_2+u \alpha_4),$ \\
$\alpha^*_3$&$=$&$x z \alpha _2+t x \alpha _3+u z \alpha _4+t u \alpha _5,$ \\
$\alpha_4^*$&$=$&$t x^2 \alpha _4,$\\
$\alpha_5^*$&$=$&$t(z \alpha_4+t \alpha_5).$\\
\end{longtable}

Since ${\rm H}_{\rm ML}^2({\mathcal N}_{01})={\rm H}_{\rm BL}^2({\mathcal N}_{01})\oplus \langle [\Delta _{23}]\rangle$ and we are interested only in new algebras, we have  $\alpha_{4}\neq 0.$
It is easy to see that, choosing
$y={x(\alpha _2 \alpha _5- \alpha _1 \alpha _4)}\alpha _4^{-2}$,
$u=-{x \alpha _2}{\alpha _4^{-1}}$ and
$z =-{t \alpha _5}{\alpha _4^{-1}},$
we have $\alpha^*_1=\alpha^*_2=\alpha^*_5=0.$ Hence, we can suppose $\alpha_1=\alpha_2=\alpha_5=0$  and we
have the following
$$\alpha^*_3=tx\alpha_3,\quad \alpha_4^*=t x^2 \alpha _4.$$

\begin{enumerate}
    \item If $\alpha_3=0,$ then we have the representative
	$\langle \nabla_4 \rangle.$
	\item If $\alpha_3\neq 0,$ then by choosing $x=\frac{\alpha_3}{\alpha_4},$  we have the representative
	$\langle \nabla_3+\nabla_4\rangle.$
\end{enumerate}

Thus, we have the following distinct orbits
$$\langle \nabla_4 \rangle, \quad \langle \nabla_3+\nabla_4\rangle,$$
which gives the following new algebras:

\begin{longtable}{llllllllllllll}
$\mathbb{M}_{01}$ & $: $ & $e_1e_1=e_2$ & $e_2e_3=e_4$ \\
$\mathbb{M}_{02}$ & $: $ & $e_1e_1=e_2$ & $e_2e_3=e_4$ & $e_3e_1=e_4$\\
\end{longtable}

\subsubsection{Central extensions of ${\mathcal N}_{02}$}
	Let us use the following notations:
	\begin{longtable}{lllllll}
	$\nabla_1 = [\Delta_{11}],$ & $\nabla_2 = [\Delta_{12}],$ &$\nabla_3 = [\Delta_{13}-\Delta_{31}],$ & $\nabla_4 = [\Delta_{22}],$ \\
	$\nabla_5 = [\Delta_{23}-\Delta_{32}],$  & $\nabla_6 = [\Delta_{31}]$, &
	$\nabla_7 = [\Delta_{32}]$, & $\nabla_8 = [\Delta_{33}].$\\
	\end{longtable}	
	
Take $\theta=\sum\limits_{i=1}^8\alpha_i\nabla_i\in {\rm H_{\rm ML}^2}({\mathcal N}_{02}).$
	The automorphism group of ${\mathcal N}_{02}$ consists of invertible matrices of the form
	$$
	\phi=
	\left(
\begin{array}{ccc}
 x & z & 0\\
 y & t & 0 \\
 u & q & xt-yz
 \end{array}
\right).$$
	
	Since
	$$
	\phi^T\begin{pmatrix}
	\alpha_1 &  \alpha_2 & \alpha_3\\
	0        & \alpha_4  & \alpha_5\\
	\alpha_6-\alpha_3& \alpha_7-\alpha_5 & \alpha_8 \\
	\end{pmatrix} \phi=
    \begin{pmatrix}
	\alpha_1^* &  \alpha_2^*-\alpha^* & \alpha_3^*\\
	\alpha^*        & \alpha_4^*  & \alpha_5^*\\
	\alpha_6^*-\alpha_3^*& \alpha_7^*-\alpha_5^* & \alpha_8^* \\
	\end{pmatrix},
	$$
	 we have that the action of ${\rm Aut} ({\mathcal N}_{02})$ on the subspace
$\langle \sum\limits_{i=1}^8\alpha_i\nabla_i  \rangle$
is given by
$\langle \sum\limits_{i=1}^8\alpha_i^{*}\nabla_i\rangle,$
where

\begin{longtable}{lcllcllcl}
$\alpha^*_1$&$=$&$x^2 \alpha_1+x y \alpha_2+y^2 \alpha_4+u x \alpha_6+u y \alpha_7+u^2 \alpha_8,$&\\
$\alpha^*_2$&$=$&$2 x z \alpha _1+(t x+y z) \alpha _2+2 t y \alpha _4+q x \alpha _6+u z \alpha _6+t u \alpha _7+q y \alpha _7+2 q u \alpha _8$\\
$\alpha^*_3$&$=$&$(xt-yz) \left(x \alpha _3+y \alpha _5+u \alpha _8\right)$,& &\\
$\alpha_4^*$&$=$&$z^2 \alpha _1+t z \alpha _2+t^2 \alpha _4+q z \alpha _6+q t \alpha _7+q^2 \alpha _8$,\\
$\alpha_5^*$&$=$&$(xt-yz) \left(z \alpha _3+t \alpha _5+q \alpha _8\right)$,& &\\
$\alpha_6^*$&$=$&$(xt-yz) \left(x \alpha _6+y \alpha _7+2u \alpha _8\right)$,\\
$\alpha_7^*$&$=$&$(xt-yz) \left(z \alpha _6+t \alpha _7+2q \alpha _8\right)$,\\
$\alpha_8^*$&$=$&$(xt-yz)^2 \alpha _8$.\\
\end{longtable}

Since ${\rm H}_{\rm ML}^2({\mathcal N}_{02})={\rm H}_{\rm BL}^2({\mathcal N}_{02})\oplus \langle [\Delta _{31}],[\Delta _{32}],[\Delta_{33}]\rangle$ and we are interested only in new algebras, we have  $(\alpha_{6},\alpha_7,\alpha_8)\neq(0,0,0).$

\begin{enumerate}
    \item  Let $\alpha_{8}=0,$ then $(\alpha_6,\alpha_7)\neq 0$ and without loss of generality (maybe with an action of a suitable $\phi$), we can suppose $\alpha_6\neq 0$. Then we consider following subcases:

  \begin{enumerate}
  \item if $\alpha_{5}=0,\alpha_4=0,\alpha_3=0,$ then choosing
 \begin{center} $x=1,$
  $y=0,$
  $u=-{\alpha_1}{\alpha^{-1}_6}$
and
$q=t (2 \alpha_1 \alpha_7-\alpha_2 \alpha_6) \alpha_6^{-2},$      \end{center}
  we have the representatives
  $\langle  \nabla_{4}+ \nabla_{6} \rangle$
  and
  $\langle   \nabla_{6} \rangle,$
  depending on $\alpha_7 (\alpha_1 \alpha_7-\alpha_2 \alpha_6)\neq 0$ or not;

  \item if $\alpha_{5}=0,\alpha_4=0,\alpha_3\neq0,\alpha_{7}=0,$ then choosing
 \begin{center}
 $x=\alpha_6,$
  $y=0,$
  $z=0,$
  $t=\alpha_6,$
  $u=-\alpha_1$
and
$q=- \alpha_2,$
 \end{center}
  we have the family of representatives
  $\langle \alpha \nabla_{3}+ \nabla_{6} \rangle_{\alpha\neq0};$

 \item if $\alpha_{5}=0,\alpha_4=0,\alpha_3\neq0,\alpha_{7}\neq0,$ then choosing
 \begin{center}
 $x=0,$
  $z=y\alpha_3^{-1}\alpha_7,$
  $t=-y\alpha_3^{-1}\alpha_6,$
  $u=0$
and
$q=- y\alpha_3^{-1}\alpha_2,$
 \end{center}
  we have the representatives
  $\langle  \nabla_{4}+ \nabla_{5}+ \nabla_{6} \rangle$
  and
  $\langle   \nabla_{5}+\nabla_{6} \rangle,$
  depending on $\alpha_1 \alpha_7\neq\alpha_2 \alpha_6 $ or not:

 \item if $\alpha_{5}=0,\alpha_4\neq0,\alpha_3=0,$ then choosing
  $z=0,$
  $t=-\alpha_7$
    and
    $q= \alpha_4,$
  we have $\alpha_4^* = \alpha_5^* = \alpha_8^*=0$ and it gives us previous considered case;

 \item  if $\alpha_{5}=0,\alpha_4\neq0,\alpha_3\neq0,$ then choosing
   $x=1,$
   $y=0,$
   $z=1,$
   and
    $t=1,$
  we have $\alpha_5^*\neq 0$ and this case will be considered below;

 \item if $\alpha_{5}\neq0, \alpha_{3}=0,$
 then choosing
  $y=0,$
  $z=-x\alpha_7\alpha_5^{-1},$
  $t=x\alpha_6\alpha_5^{-1},$
$u=-x\alpha_1\alpha_6^{-1}$
and
$q=-x (2 \alpha_1 \alpha_7-\alpha_2 \alpha_6)\alpha_5^{-1} \alpha_6^{-1},$
  we have the  representatives
  $\langle \nabla_{4}+ \nabla_{5}+ \nabla_{6} \rangle$
  and
  $\langle   \nabla_{5}+\nabla_{6} \rangle,$
  depending on $\alpha_4 \alpha_6^2\neq\alpha_7 (\alpha_2 \alpha_6-\alpha_1 \alpha_7)$ or not;

  \item if $\alpha_{5}\neq0, \alpha_{3}\neq0,
\alpha_5 \alpha_6=\alpha_3 \alpha_7,$
  then choosing
   $x=\alpha_6,$
   $y=0,$
  $z=-t \alpha_5 \alpha_3^{-1},$
  $u=-\alpha_1$
    and
    $q=-t(2\alpha_1 \alpha_5-  \alpha_2 \alpha_3)\alpha_3^{-1} \alpha_6^{-1},$
  we have the families  of representatives
  $\langle  \alpha \nabla_{3}+\nabla_{4}+   \nabla_{6} \rangle_{\alpha\neq 0}$
  and
  $\langle  \alpha \nabla_{3}+   \nabla_{6} \rangle_{\alpha\neq 0},$
  depending on
  $\alpha_3^2 \alpha_4+\alpha_1 \alpha_5^2\neq\alpha_2 \alpha_3 \alpha_5$ or not;

     \item if $\alpha_{5}\neq0, \alpha_{3}\neq0,
\alpha_5 \alpha_6\neq \alpha_3 \alpha_7,$
  then by choosing
   $x=-y\alpha_5\alpha_3^{-1},$
   we have $\alpha_3^*=0,$ which gives the previously considered case.
    \end{enumerate}

\item Let $\alpha_{8}\neq0$. Then putting
$u=-\frac{x \alpha _6+y \alpha _7}{2 \alpha _8},$ $q=-\frac{z \alpha _6+t \alpha _7}{2 \alpha _8},$
we have $\alpha^*_6=\alpha^*_7=0.$
Thus, without loss of generality, we can suppose $\alpha_6=\alpha_7=0$ and obtain
\begin{longtable}{lcllcl}
$\alpha^*_1$&$=$&$x^2 \alpha _1+xy\alpha _2+y^2 \alpha_4,$ &
$\alpha^*_2$&$=$&$2 x z \alpha _1+(t x+y z) \alpha _2+2 t y \alpha _4,$\\
$\alpha^*_3$&$=$&$(t x-y z) \left(x \alpha _3+y \alpha _5\right),$ &
$\alpha_4^*$&$=$&$z^2 \alpha _1+zt\alpha_2+t^2\alpha_4,$\\
$\alpha_5^*$&$=$&$(t x-y z) \left(z \alpha _3+t \alpha _5\right),$ &
$\alpha_8^*$&$=$&$(xt-yz)^2 \alpha_8.$
\end{longtable}

\begin{enumerate}
    \item Let $(\alpha_{3},\alpha_{5})=(0,0)$. Then we have the following subcases:

%$$\begin{array}{lllll}
%\alpha^*_1&=&x^2 \alpha_1+yx\alpha _2+y^2\alpha_4\\
%\alpha^*_2&=&2 x z \alpha _1+(t x+y z) \alpha _2+2 t y \alpha _4,\\
%\alpha^*_3&=&0,\\
%\alpha_4^*&=&z^2 \alpha _1+tz\alpha _2+t^2\alpha _4,\\
%\alpha_5^*&=&0,\\
%\alpha_8^*&=&(xt-yz)^2 \alpha_8.
%\end{array}$$

\begin{enumerate}
\item $(\alpha_{1},\alpha_2,\alpha_4)=(0,0,0),$ then we have the representative $\langle \nabla_{8} \rangle$.

\item $(\alpha_{1},\alpha_2,\alpha_4)\neq (0,0,0),$ then
without loss of generality (maybe with an action of a suitable $\phi$),
we can suppose $\alpha_1\neq 0$.

  \begin{enumerate}
    \item $\alpha _2^2\neq4\alpha_1 \alpha_4,$ then choosing $x=-\frac{\alpha _2+\sqrt{\alpha _2^2-4 \alpha _1 \alpha _4}}{2 \alpha _1},$ $y=1,$ $z=\frac{-\alpha _2+\sqrt{\alpha _2^2-4 \alpha _1 \alpha _4}}{2 \alpha_8}$, $t=\frac {\alpha_1}{\alpha_8},$ we have
the representative $\langle \nabla_{2}+\nabla_{8} \rangle.$

    \item $\alpha _2^2=4\alpha_1 \alpha_4,$ then choosing $z=-\frac{\alpha _2}{2 \alpha _1}\sqrt{\frac{\alpha_1}{{\alpha_8}}}$, $t= \sqrt{\frac{\alpha_1}{{\alpha_8}}},$ we have the representative
$\langle \nabla_{1}+\nabla_{8} \rangle.$
 \end{enumerate}

\end{enumerate}

\item  Let $(\alpha_{3},\alpha_{5})\neq(0,0)$, then
without loss of generality (maybe with an action of a suitable $\phi$)
one can suppose $\alpha_3\neq 0$ and choosing $z=-\frac{t \alpha_5}{\alpha_3},$ we have
$\alpha^*_5=0$.

%$$\begin{array}{lllll}
%\alpha^*_1&=&x^2 \alpha_1+xy\alpha_2+y^2 \alpha_4\\
%\alpha^*_2&=&t(x \alpha_2+2 y \alpha_4),\\
%\alpha^*_3&=&t x^2 \alpha _3,\\
%\alpha_4^*&=&t^2 \alpha _4,\\
%\alpha_8^*&=&x^2t^2\alpha_8.
%\end{array}$$

\begin{enumerate}
\item $\alpha_4=0,\alpha_2=0,$ then choosing $t=\alpha_3\alpha_8^{-1},$ we have the family of  representatives $\langle \alpha  \nabla_{1}+\nabla_{3}+\nabla_{8} \rangle$;

\item $\alpha_4=0,\alpha_2\neq 0,$ then choosing
$x=\alpha_2\alpha_3^{-1},$
$y=-\alpha_1\alpha_2^{-1},$
$t=\alpha_3\alpha_8^{-1},$ we have the representative $\langle \nabla_{2}+\nabla_{3}+\nabla_{8} \rangle$.

\item $\alpha_4\neq 0,$ then choosing
 $x=\sqrt{\frac{\alpha_4}{\alpha_8}},$
 $y=-\frac{ \alpha _2}{2 \sqrt{\alpha _4\alpha_8}}$
 and
$t=\frac{\alpha_3}{\alpha_8},$ we have the family of  representatives $\langle \alpha \nabla_{1}+\nabla_{3}+\nabla_{4}+\nabla_{8} \rangle$.
\end{enumerate}
\end{enumerate}
\end{enumerate}

Summarizing, all considered case we have the following distinct orbits
\begin{center}
$\langle \alpha\nabla_3+\nabla_6\rangle,$ \
$\langle \alpha\nabla_3+\nabla_4+\nabla_6\rangle,$ \
$\langle \nabla_5+\nabla_6\rangle,$ \
$\langle \nabla_4+\nabla_5+\nabla_6\rangle,$ \
$\langle \nabla_8\rangle,$ \
$\langle \nabla_2+\nabla_8\rangle,$\
$\langle \nabla_1+\nabla_8\rangle,$ \
$\langle \alpha\nabla_1+\nabla_3+\nabla_8\rangle,$ \
$\langle \nabla_2+\nabla_3+\nabla_8\rangle,$\
$\langle \alpha\nabla_1+\nabla_3+\nabla_4+\nabla_8\rangle,$
\end{center}
which gives the following new algebras:
\begin{longtable}{llllllllllllll}
$\mathbb{M}_{03}^\alpha$ & $: $ & $e_1e_2=e_3$ & $e_1e_3=\alpha e_4$ & $e_2e_1=-e_3$ & \multicolumn{2}{l}{ $e_3e_1=(1-\alpha)e_4$} \\
$\mathbb{M}_{04}^\alpha$ & $: $ & $e_1e_2=e_3$ & $e_1e_3=\alpha e_4$ & $e_2e_1=-e_3$ & $e_2e_2=e_4$ & \multicolumn{2}{l}{$e_3e_1=(1-\alpha)e_4$} \\
$\mathbb{M}_{05}$ & $: $ & $e_1e_2=e_3$ & $e_2e_1=-e_3$ & $e_2e_3=e_4$ & $e_3e_1=e_4$ & $e_3e_2=-e_4$ \\
$\mathbb{M}_{06}$ & $: $ & $e_1e_2=e_3$ & $e_2e_1=-e_3$ & $e_2e_2=e_4$ & $e_2e_3=e_4$ & $e_3e_1=e_4$ & $e_3e_2=-e_4$ \\
$\mathbb{M}_{07}$ & $: $ & $e_1e_2=e_3$ & $e_2e_1=-e_3$ & $e_3e_3=e_4$ \\
$\mathbb{M}_{08}$ & $: $ & $e_1e_2=e_3+e_4$ & $e_2e_1=-e_3$ & $e_3e_3=e_4$ \\
$\mathbb{M}_{09}$ & $: $ & $e_1e_1=e_4$ & $e_1e_2=e_3$ & $e_2e_1=-e_3$ & $e_3e_3=e_4$ \\
$\mathbb{M}_{10}^\alpha$ & $: $ & $e_1e_1=\alpha e_4$ & $e_1e_2=e_3$ & $e_1e_3=e_4$ & $e_2e_1=-e_3$ & $e_3e_1=-e_4$ & $e_3e_3=e_4$ \\
$\mathbb{M}_{11}$ & $: $ & $e_1e_2=e_3+e_4$ & $e_1e_3=e_4$ & $e_2e_1=-e_3$ & $e_3e_1=-e_4$ & $e_3e_3=e_4$ \\
$\mathbb{M}_{12}^\alpha$ & $: $ & $e_1e_1=\alpha e_4$ & $e_1e_2=e_3$ & $e_1e_3=e_4$ & $e_2e_1=-e_3$ \\
& & $e_2e_2=e_4$& $e_3e_1=-e_4$ & $e_3e_3=e_4$ \\
\end{longtable}

\subsubsection{Central extensions of ${\mathcal N}_{03}^{\alpha\neq0}$}
	Let us use the following notations:
	\begin{longtable}{lllllll}
	$\nabla_1 = [\Delta_{12}],$ & $\nabla_2 = [\Delta_{21}],$ &$\nabla_3 = [\Delta_{22}],$ & $\nabla_4 = [\Delta_{31}+\Delta_{32}],$ & $\nabla_5 = [\Delta_{32}].$
	\end{longtable}	
	
Take $\theta=\sum\limits_{i=1}^5\alpha_i\nabla_i\in {\rm H_{\rm ML}^2}({\mathcal N}_{03}^{\alpha \neq0 }).$
	The automorphism group of ${\mathcal N}_{03}^{\alpha \neq 0}$ consists of invertible matrices of the form
	$$\phi=
	\begin{pmatrix}
	x &  -\alpha y  & 0\\
	y &  x+y & 0\\
	z &  t  & x^2+xy+\alpha y^2\\
    \end{pmatrix}.$$
	
 Since
	$$
	\phi^T\begin{pmatrix}
	0 & \alpha_1  & 0\\
	\alpha_2  & \alpha_3 & 0\\
	\alpha_4 &  \alpha_4+\alpha_5 & 0\\
	\end{pmatrix} \phi=	\begin{pmatrix}
	\alpha^* & \alpha_1^*+\alpha^* & 0\\
	\alpha_2^*  & \alpha_3^*+\alpha \alpha^* & 0\\
	\alpha_4^* &  \alpha_4^*+\alpha_5^* & 0\\
	\end{pmatrix},
	$$
	 we have that the action of ${\rm Aut} ({\mathcal N}_{03}^{\alpha \neq 0})$ on the subspace
$\langle \sum\limits_{i=1}^5\alpha_i\nabla_i  \rangle$
is given by
$\langle \sum\limits_{i=1}^5\alpha_i^{*}\nabla_i\rangle,$
where
\begin{longtable}{lcl}
$\alpha^*_1$&$=$& $x^2\alpha_1-y(x+y\alpha)\alpha_2+xy\alpha_3-yz\alpha\alpha_4+xz\alpha _5,$ \\
$\alpha^*_2$&$=$&$-y^2\alpha\alpha_1+(x+y)(x\alpha_2+y\alpha_3)+t(x+y)\alpha _4+ty \alpha_5,$ \\
$\alpha^*_3$&$=$&$-y(2x+y))\alpha(\alpha_1+\alpha_2)+((x+y)^2-y^2\alpha)\alpha_3+$\\
&&$(t (x+y)-(t y+x z+y z) \alpha)\alpha_4+(t(x+y)-yz\alpha)\alpha _5,$\\
$\alpha_4^*$&$=$&$(x^2+x y+y^2 \alpha)((x+y)\alpha _4+y \alpha _5),$\\
$\alpha_5^*$&$=$&$(x^2+x y+y^2\alpha)(x\alpha_5-y\alpha \alpha _4).$\\
\end{longtable}

Since ${\rm H}_{\rm ML}^2({\mathcal N}_{03}^{\alpha\neq 0})={\rm H}_{\rm BL}^2({\mathcal N}_{03}^{\alpha\neq 0})\oplus \langle [\Delta_{31}], [\Delta_{32}]\rangle$ and we are interested only in new algebras, we have  $(\alpha_{4},\alpha_5) \neq (0,0).$ Without loss of generality, one can assume $\alpha_4\neq0.$

\begin{enumerate}
    \item $\alpha\alpha _4^2+\alpha_4\alpha_5+\alpha_5^2\neq 0,$ then choosing
\begin{center}    $x=-\frac{y(\alpha _4+\alpha_5)}{\alpha _4},\ t=\frac{y(\alpha  \alpha _1 \alpha _4(\alpha _4+2 \alpha_5)+\alpha \alpha _2 \alpha _4(\alpha_4+2\alpha_5)+\alpha_3(\alpha _5^2-\alpha  \alpha _4^2))}{\alpha _4(\alpha  \alpha _4^2+\alpha _4 \alpha _5+\alpha _5^2)},$
    $z=\frac{y(\alpha_1(\alpha _4+\alpha_5)^2-\alpha _4(\alpha _2((\alpha-1 ) \alpha _4-\alpha _5)+\alpha _3(\alpha _4+\alpha _5)))}{\alpha _4(\alpha  \alpha _4^2+\alpha _4 \alpha _5+\alpha _5^2)},$\end{center}
we have
\begin{center}
    $\alpha^*_1=\alpha^*_3=\alpha^*_4=0,$ $\alpha^*_2=\frac{y^2(\alpha_5(\alpha _2(\alpha _4+\alpha _5)-\alpha_3 \alpha _4)-\alpha\alpha_1\alpha_4^2)}{\alpha _4^2},$ $\alpha^*_5=-\frac{y^3(\alpha  \alpha _4^2+\alpha _4 \alpha _5+\alpha _5^2){}^2}{\alpha _4^3}.$\end{center}

Hence, we have two representatives
$\langle \nabla_5\rangle$ and $\langle \nabla_2+ \nabla_5\rangle,$
depending on
$\alpha\alpha_1\alpha_4^2=\alpha_5(\alpha _2(\alpha _4+\alpha _5)-\alpha_3 \alpha _4)$ or not.

    \item $\alpha\alpha _4^2+\alpha_4\alpha_5+\alpha_5^2=0,$ i.e., $\alpha_5=-\frac{1}{2}(1 \pm \sqrt{1-4\alpha})\alpha_4,$ then choosing
\begin{center}
$y=0,$
$t=-\frac{x \alpha _2}{\alpha _4},$
$z=\frac{2x\alpha _1}{\alpha _4(1\pm \sqrt{1-4\alpha })},$
\end{center}    we have
\begin{center}
    $\alpha^*_1=\alpha^*_2=0,$ \
    $\alpha^*_3=-\frac{x^2(2\alpha \alpha _1 +2 \alpha  \alpha _2 -\alpha _3(1\pm \sqrt{1-4\alpha }))}{1\pm \sqrt{1-4\alpha}},$ \
    $\alpha^*_4=x^3 \alpha _4,$  \
    $\alpha^*_5=-\frac{x^3\alpha _4(1\pm \sqrt{1-4\alpha})}{2} .$
    \end{center}

Hence, we have four families of representatives
\begin{center}$\langle \nabla_3+\nabla_4-\frac{1}{2}(1\pm \sqrt{1-4\alpha}) \nabla_5 \rangle$ and
$\langle \nabla_4-\frac{1}{2}(1\pm \sqrt{1-4\alpha}) \nabla_5 \rangle,$
\end{center}
depending on
$2\alpha\alpha_4( \alpha_1 +\alpha _2)\neq\alpha _3(\alpha_4 \pm \sqrt{(1-4\alpha)\alpha _4^2})$ or not.

\end{enumerate}

\subsubsection{Central extensions of ${\mathcal N}_{03}^{0}$}
	Let us use the following notations:
	\begin{longtable}{lllllll}
	$\nabla_1 = [\Delta_{12}],$ & $\nabla_2 = [\Delta_{21}],$ &$\nabla_3 = [\Delta_{22}],$ & $\nabla_4 = [\Delta_{31}+\Delta_{32}],$ & $\nabla_5 = [\Delta_{32}]$
	\end{longtable}	
	
Take $\theta=\sum\limits_{i=1}^5\alpha_i\nabla_i\in {\rm H_{\rm ML}^2}({\mathcal N}_{03}^{0 }).$
	The automorphism group of ${\mathcal N}_{03}^{0}$ consists of invertible matrices of the form
	$$\phi=
	\begin{pmatrix}
	x &  0  & 0\\
	y-x &  y & 0\\
	z &  t  & xy\\
    \end{pmatrix}.$$

 Since
	$$
	\phi^T\begin{pmatrix}
	0 & \alpha_1  & 0\\
	\alpha_2  & \alpha_3 & 0\\
	\alpha_4 & \alpha_4+\alpha_5 & 0\\
	\end{pmatrix} \phi=	\begin{pmatrix}
	\alpha^* & \alpha_1^*+\alpha^* & 0\\
	\alpha_2^*  & \alpha_3^*& 0\\
	\alpha_4^* & \alpha_4^*+\alpha_5^* & 0\\
	\end{pmatrix},
	$$
	 we have that the action of ${\rm Aut} ({\mathcal N}_{03}^{0})$ on the subspace
$\langle \sum\limits_{i=1}^5\alpha_i\nabla_i  \rangle$
is given by
$\langle \sum\limits_{i=1}^5\alpha_i^{*}\nabla_i\rangle,$
where
\begin{longtable}{lcl}
$\alpha^*_1$&$=$& $x(x\alpha_1+(x-y)\alpha _2-x \alpha _3+y \alpha _3+z \alpha_5),$ \\
$\alpha^*_2$&$=$&$x y \alpha _2+y (y-x) \alpha _3+t(y\alpha _4+(y-x) \alpha_5),$ \\
$\alpha^*_3$&$=$&$y(y \alpha _3+t(\alpha _4+\alpha_5)),$\\
$\alpha_4^*$&$=$&$xy(y \alpha_4+(y-x) \alpha_5),$\\
$\alpha_5^*$&$=$&$x^2y\alpha_5.$\\
\end{longtable}

Since ${\rm H}_{\rm ML}^2({\mathcal N}_{03}^{0})={\rm H}_{\rm BL}^2({\mathcal N}_{03}^{0})\oplus \langle [\Delta_{32}]\rangle$ and we are interested only in new algebras, we have  $\alpha_5\neq 0.$ Then choosing  $z=(y(\alpha_2-\alpha_3)-x(\alpha_1+\alpha _2-\alpha_3))\alpha _5^{-1},$ we have $\alpha^*_1=0.$

\begin{enumerate}
    \item If $\alpha_5+\alpha_4=0,$ then choosing
    $t=y \left(x \alpha _2+(y-x) \alpha _3\right)x^{-1} \alpha^{-1} _5,$ we have
\begin{center}
    $\alpha^*_1=\alpha^*_2=0,$\
    $\alpha^*_3=y^2\alpha_3,$\
    $\alpha^*_4=-x^2y\alpha_5,$\
    $\alpha^*_5=x^2y\alpha_5.$
\end{center}
Hence, we have two representatives
$\langle \nabla_4-\nabla_5\rangle$ and
$\langle \nabla_3+\nabla_4- \nabla_5\rangle,$
depending on
$\alpha_3=0$ or not.

     \item If $\alpha_5+\alpha_4\neq 0,$ then choosing
     $y=\frac{x \alpha _5}{\alpha _4 +\alpha_5}, \ t=-\frac{x\alpha_3 \alpha_5}{(\alpha _4+\alpha_5)^2},$ we have
     \begin{center}
     $\alpha^*_1=0,$ \
     $\alpha^*_2=\frac{x^2\alpha_2 \alpha_5}{\alpha _4+\alpha_5},$ \ $\alpha^*_3=\alpha_4^*=0,$\
     $\alpha^*_5=\frac{x^3\alpha_5^2}{\alpha _4+\alpha_5}.$\end{center}
Hence, we have two representatives
$\langle \nabla_5\rangle$ and
$\langle \nabla_2+\nabla_5\rangle,$
depending on
$\alpha_2=0$ or not.
\end{enumerate}
Summarizing all cases of the central extension of the algebra ${\mathcal N}_{03}^{\alpha}$, we have the following distinct orbits
\begin{center}
$\langle \nabla_5\rangle,$ \
$\langle \nabla_2+\nabla_5\rangle, \quad \langle \nabla_4-\frac{1}{2}(1+\sqrt{1-4\alpha})\nabla_5\rangle,$ \
$\langle \nabla_3+\nabla_4 -\frac{1}{2}(1+\sqrt{1-4\alpha})\nabla_5\rangle,$ \
$\langle \nabla_4-\frac{1}{2}(1-\sqrt{1-4\alpha})\nabla_5\rangle_{\alpha\neq 0},$ \
$\langle \nabla_3+\nabla_4 -\frac{1}{2}(1-\sqrt{1-4\alpha})\nabla_5\rangle_{\alpha\neq0},$
\end{center}
which gives the following new algebras:

\begin{longtable}{llllllllllllll}
$\mathbb{M}_{13}^\alpha$ & $: $ & $e_1e_1=e_3$ & $e_1e_2=e_3$ & $e_2e_2=\alpha e_3$ & $e_3e_2=e_4$ \\
$\mathbb{M}_{14}^\alpha$ & $: $ & $e_1e_1=e_3$ & $e_1e_2=e_3$ & $e_2e_1=e_4$ & $e_2e_2=\alpha e_3$ & $e_3e_2=e_4$ \\
$\mathbb{M}_{15}^\alpha$ & $: $ & $e_1e_1=e_3$ & $e_1e_2=e_3$ & $e_2e_2=\alpha e_3$ & $e_3e_1=e_4$ & $e_3e_2=\frac{1}{2}(1-\sqrt{1-4\alpha})e_4$ \\
$\mathbb{M}_{16}^\alpha$ & $: $ & $e_1e_1=e_3$ & $e_1e_2=e_3$ & $e_2e_2=\alpha e_3+e_4$ & $e_3e_1=e_4$ & $e_3e_2=\frac{1}{2}(1-\sqrt{1-4\alpha})e_4$ \\
$\mathbb{M}_{17}^{\alpha\neq 0}$ & $:$ & $e_1e_1=e_3$ & $e_1e_2=e_3$ & $e_2e_2=\alpha e_3$ & $e_3e_1=e_4$ & $e_3e_2=\frac{1}{2}(1+\sqrt{1-4\alpha})e_4$ \\
$\mathbb{M}_{18}^{\alpha\neq 0}$ & $:$ & $e_1e_1=e_3$ & $e_1e_2=e_3$ & $e_2e_2=\alpha e_3+e_4$ & $e_3e_1=e_4$ & $e_3e_2=\frac{1}{2}(1+\sqrt{1-4\alpha})e_4$ \\
\end{longtable}

Note that algebras $\mathbb{M}_{17}^{4}(0)$ and $\mathbb{M}_{18}^{4}(0)$ are Leibniz algebras with one dimensional annihilator.

\subsubsection{Central extensions of ${\mathcal N}_{04}$}
	Let us use the following notations:
	\begin{longtable}{lllllll}
	$\nabla_1 = [\Delta_{12}],$ & $\nabla_2 = [\Delta_{21}],$ &$\nabla_3 = [\Delta_{22}],$ & $\nabla_4 = [\Delta_{31}],$ & $\nabla_5 = [\Delta_{32}].$
	\end{longtable}	
	
Take $\theta=\sum\limits_{i=1}^5\alpha_i\nabla_i\in {\rm H_{\rm ML}^2}({\mathcal N}_{04}).$
	The automorphism group of ${\mathcal N}_{04}$ consists of invertible matrices of the form
	$$\phi_1=
	\begin{pmatrix}
	x & y  & 0\\
	-y & x & 0\\
	z &  t  & x^2+y^2\\
    \end{pmatrix}, \quad \phi_2=
	\begin{pmatrix}
	x &  y  & 0\\
	y & -x & 0\\
	z &  t  & x^2+y^2\\
    \end{pmatrix}.$$

 Since
	$$
	\phi_1^T\begin{pmatrix}
	0 & \alpha_1  & 0\\
	\alpha_2  & \alpha_3 & 0\\
	\alpha_4 & \alpha_5 & 0\\
	\end{pmatrix} \phi_1=	\begin{pmatrix}
	\alpha^* & \alpha_1^* & 0\\
	\alpha_2^* & \alpha_3^*+\alpha^*& 0\\
	\alpha_4^* &\alpha_5^* & 0\\
	\end{pmatrix},
	$$
	 we have that the action of ${\rm Aut} ({\mathcal N}_{04})$ on the subspace
$\langle \sum\limits_{i=1}^5\alpha_i\nabla_i  \rangle$
is given by
$\langle \sum\limits_{i=1}^5\alpha_i^{*}\nabla_i\rangle,$
where
\begin{longtable}{lcl}
$\alpha^*_1$&$=$& $x^2 \alpha _1-y(y \alpha _2-z \alpha_4)-x(y \alpha _3-z \alpha _5),$ \\
$\alpha^*_2$&$=$&$-y^2 \alpha _1+x(x\alpha _2-y \alpha _3+t\alpha_4)-t y \alpha _5,$ \\
$\alpha^*_3$&$=$&$2xy(\alpha_1+\alpha_2)+(x^2-y^2)\alpha_3+(ty-xz)\alpha _4+(tx+yz)\alpha _5,$\\
$\alpha_4^*$&$=$&$(x^2+y^2)(x \alpha _4-y \alpha _5),$\\
$\alpha_5^*$&$=$&$(x^2+y^2)(y\alpha _4+x\alpha_5).$\\
\end{longtable}

Since ${\rm H}_{\rm ML}^2({\mathcal N}_{04})={\rm H}_{\rm BL}^2({\mathcal N}_{04})\oplus \langle [\Delta_{31}], [\Delta_{32}]\rangle$ and we are interested only in new algebras, we have  $(\alpha_4,\alpha_5)\neq (0,0).$ Moreover, without loss of generality, one can assume $\alpha_4\neq 0$. Then we have the following cases.

\begin{enumerate}
    \item If $\alpha_4^2+\alpha_5^2\neq 0,$ then choosing
    \begin{center}
        $y=-\frac{x \alpha _5}{\alpha _4},$ \
        $t=-\frac{x \alpha _2 \alpha _4^2+x \alpha _3 \alpha _4 \alpha _5-x \alpha _1 \alpha _5^2}{\alpha _4(\alpha _4^2+\alpha_5^2)}$, \  $z=\frac{x(\alpha _3 (\alpha _4^2-\alpha _5^2)-2(\alpha _1+\alpha _2) \alpha _4 \alpha _5)}{\alpha _4(\alpha _4^2+\alpha _5^2)},$
    \end{center}we have
\begin{center}
 $\alpha_1^*=\frac{x^2(\alpha _1 \alpha _4^2+\alpha _5 (\alpha _3 \alpha _4-\alpha_2\alpha_5))}{\alpha _4^2},$\
 $\alpha^*_2=\alpha_3^*=0,$\
 $\alpha^*_4=\frac{x^3(\alpha _4^2+\alpha _5^2)^2}{\alpha _4^3},$
$\alpha^*_5=0.$
\end{center}
Hence, we have two representatives
$\langle \nabla_4\rangle$ and
$\langle \nabla_1+\nabla_4\rangle,$
depending on
$\alpha _1 \alpha _4^2+\alpha _5 \alpha _3 \alpha _4=\alpha_2\alpha_5^2$ or not.

\item If $\alpha_4^2+\alpha_5^2=0,$ then choosing
\begin{center}
$t=\frac{y^2 \alpha _1-x^2 \alpha _2+x y \alpha _3}{ \alpha_4(x\pm i y)}, \ z=\frac{y \alpha _1\alpha _4((2 x^2+y^2)\pm i x y)+\alpha _3\alpha _4(x^3\mp iy^3)+x \alpha _2\alpha _4(xy\pm i(x^2+2 y^2))}{\alpha_4^2(x\pm i y)^2},$
\end{center}

we have
\begin{center}
$\alpha_1^*=\frac{(x^2+y^2)^2(\alpha _1+\alpha _2\mp \alpha _3)}{(x\pm i y)^2},$ \
$\alpha^*_2=\alpha_3^*=0,$\
 $\alpha^*_4=(x^2+y^2)(x\pm i y) \alpha _4,\ \alpha_5^*=\mp i (x^2+y^2)(x\mp i y)\alpha_4.$
\end{center}

Hence, we have two representatives
$\langle   \nabla_4\pm i\nabla_5\rangle$ and
$\langle  \nabla_1+\nabla_4\pm i\nabla_5\rangle,$
depending on
$\alpha _1+\alpha _2\mp \alpha _3=0$ or not.

Since the automorphism $\phi_2 =diag(1,-1, 1)$ acts as \begin{center}
$\phi_2(\nabla_4+i\nabla_5) =\nabla_4-i\nabla_5$	 and $\phi_2(\nabla_1+\nabla_4+i\nabla_5) =\nabla_1+\nabla_4-i\nabla_5,$
\end{center}
we get the following representatives of distinct orbits $\langle \nabla_4+i\nabla_5\rangle$ and $\langle \nabla_1+\nabla_4+i\nabla_5\rangle.$
\end{enumerate}

Summarizing, all considered cases we have the following distinct orbits
\begin{longtable} {llll}
$\langle \nabla_4\rangle,$ & $\langle \nabla_1+\nabla_4\rangle,$ & $\langle \nabla_4+i\nabla_5\rangle,$ &
$\langle \nabla_1+\nabla_4+i\nabla_5\rangle,$ \\
\end{longtable}
which gives the following new algebras:
\begin{longtable}{llllllllllllll}
$\mathbb{M}_{19}$ & $: $ & $e_1e_1=e_3$ & $e_2e_2=e_3$ & $e_3e_1=e_4$\\
$\mathbb{M}_{20}$ & $: $ & $e_1e_1=e_3$ & $e_1e_2=e_4$ & $e_2e_2=e_3$ & $e_3e_1=e_4$ \\
$\mathbb{M}_{21}$ & $: $ & $e_1e_1=e_3$ & $e_2e_2=e_3$ & $e_3e_1=e_4$ & $e_3e_2=i e_4$\\
$\mathbb{M}_{22}$ & $: $ & $e_1e_1=e_3$ & $e_1e_2=e_4$ & $e_2e_2=e_3$ & $e_3e_1=e_4$ & $e_3e_2=i e_4$\\
\end{longtable}

Now we are ready to summarize all results related to the algebraic classification of complex $4$-dimensional nilpotent mono  Leibniz algebras.

\subsubsection{The classification theorem}\label{secteoB}

\begin{theoremB}
Let ${\mathbb L}$ be a complex  $4$-dimensional nilpotent mono  Leibniz algebra.
Then ${\mathbb L}$ is a binary Leibniz  algebra or isomorphic to one algebra from the following list:

{\small
\begin{longtable}{llllllllllllll}
$\mathbb{M}_{01}$ & $: $ & $e_1e_1=e_2$ & $e_2e_3=e_4$ \\
$\mathbb{M}_{02}$ & $: $ & $e_1e_1=e_2$ & $e_2e_3=e_4$ & $e_3e_1=e_4$\\
$\mathbb{M}_{03}^\alpha$ & $: $ & $e_1e_2=e_3$ & $e_1e_3=\alpha e_4$ & $e_2e_1=-e_3$   & \multicolumn{2}{l}{$e_3e_1=(1-\alpha)e_4$} \\
$\mathbb{M}_{04}^\alpha$ & $: $ & $e_1e_2=e_3$ & $e_1e_3=\alpha e_4$ & $e_2e_1=-e_3$ & $e_2e_2=e_4$ & $e_3e_1=(1-\alpha)e_4$ \\
$\mathbb{M}_{05}$ & $: $ & $e_1e_2=e_3$ & $e_2e_1=-e_3$ & $e_2e_3=e_4$ & $e_3e_1=e_4$ & $e_3e_2=-e_4$ \\
$\mathbb{M}_{06}$ & $: $ & $e_1e_2=e_3$ & $e_2e_1=-e_3$ & $e_2e_2=e_4$ \\
&& $e_2e_3=e_4$ & $e_3e_1=e_4$ & $e_3e_2=-e_4$ \\
$\mathbb{M}_{07}$ & $: $ & $e_1e_2=e_3$ & $e_2e_1=-e_3$ & $e_3e_3=e_4$ \\
$\mathbb{M}_{08}$ & $: $ & $e_1e_2=e_3+e_4$ & $e_2e_1=-e_3$ & $e_3e_3=e_4$ \\
$\mathbb{M}_{09}$ & $: $ & $e_1e_1=e_4$ & $e_1e_2=e_3$ & $e_2e_1=-e_3$ & $e_3e_3=e_4$ \\
$\mathbb{M}_{10}^\alpha$ & $: $ & $e_1e_1=\alpha e_4$ & $e_1e_2=e_3$ & $e_1e_3=e_4$ \\
&& $e_2e_1=-e_3$   & $e_3e_1=-e_4$ & $e_3e_3=e_4$ \\
$\mathbb{M}_{11}$ & $: $ & $e_1e_2=e_3+e_4$ & $e_1e_3=e_4$ & $e_2e_1=-e_3$ & $e_3e_1=-e_4$ & $e_3e_3=e_4$ \\
$\mathbb{M}_{12}^\alpha$ & $: $ & $e_1e_1=\alpha e_4$ & $e_1e_2=e_3$ & $e_1e_3=e_4$ & $e_2e_1=-e_3$\\
&& $e_2e_2=e_4$  & $e_3e_1=-e_4$ & $e_3e_3=e_4$ \\
$\mathbb{M}_{13}^\alpha$ & $: $ & $e_1e_1=e_3$ & $e_1e_2=e_3$ & $e_2e_2=\alpha e_3$ & $e_3e_2=e_4$ \\
$\mathbb{M}_{14}^\alpha$ & $: $ & $e_1e_1=e_3$ & $e_1e_2=e_3$ & $e_2e_1=e_4$ & $e_2e_2=\alpha e_3$ & $e_3e_2=e_4$ \\
$\mathbb{M}_{15}^\alpha$ & $: $ & $e_1e_1=e_3$ & $e_1e_2=e_3$ & $e_2e_2=\alpha e_3$ & $e_3e_1=e_4$ & $e_3e_2=\frac{1}{2}(1-\sqrt{1-4\alpha})e_4$ \\
$\mathbb{M}_{16}^\alpha$ & $: $ & $e_1e_1=e_3$ & $e_1e_2=e_3$ & $e_2e_2=\alpha e_3+e_4$ & $e_3e_1=e_4$ & $e_3e_2=\frac{1}{2}(1-\sqrt{1-4\alpha})e_4$ \\
$\mathbb{M}_{17}^{\alpha\neq 0}$ & $:$ & $e_1e_1=e_3$ & $e_1e_2=e_3$ & $e_2e_2=\alpha e_3$ & $e_3e_1=e_4$ & $e_3e_2=\frac{1}{2}(1+\sqrt{1-4\alpha})e_4$ \\
$\mathbb{M}_{18}^{\alpha\neq 0}$ & $:$ & $e_1e_1=e_3$ & $e_1e_2=e_3$ & $e_2e_2=\alpha e_3+e_4$ & $e_3e_1=e_4$ & $e_3e_2=\frac{1}{2}(1+\sqrt{1-4\alpha})e_4$ \\
$\mathbb{M}_{19}$ & $: $ & $e_1e_1=e_3$ & $e_2e_2=e_3$ & $e_3e_1=e_4$\\
$\mathbb{M}_{20}$ & $: $ & $e_1e_1=e_3$ & $e_1e_2=e_4$ & $e_2e_2=e_3$ & $e_3e_1=e_4$ \\
$\mathbb{M}_{21}$ & $: $ & $e_1e_1=e_3$ & $e_2e_2=e_3$ & $e_3e_1=e_4$ & $e_3e_2=i e_4$\\
$\mathbb{M}_{22}$ & $: $ & $e_1e_1=e_3$ & $e_1e_2=e_4$ & $e_2e_2=e_3$ & $e_3e_1=e_4$ & $e_3e_2=i e_4$\\
\end{longtable}}
Note that algebras $\mathbb{M}_{17}^{4}(0)$ and $\mathbb{M}_{18}^{4}(0)$ are Leibniz algebras.
\end{theoremB}

\subsection{Classification of $4$-dimensional nilpotent algebras with nil-index  $3$}\label{secteoC}
Thanks to \cite{bkm22}, the intersection of left mono Leibniz and right mono Leibniz algebras gives the
variety of nil-algebras of nil-index $3.$
Hence, each $4$-dimensional nilpotent algebra with nil-index  $3$ is in the classification given in Theorem B.
Obviously,
each $2$-step nilpotent algebra has nil-index $3.$
Using the classification of $4$-dimensional Leibniz algebras \cite{kppv} and Theorem B,
we can choose all Leibniz algebras with nil-index $3.$
Let us note that the linearization of the identity $x^3=0$ (i.e. $x^2x=0=xx^2$)
gives two identities \begin{center}
$\sum\limits_{\sigma \in \mathbb S_3} (x_{\sigma(1)}x_{\sigma(2)})x_{\sigma(3)}=0$
and
$\sum\limits_{\sigma \in \mathbb S_3} x_{\sigma(1)}(x_{\sigma(2)}x_{\sigma(3)})=0$.
\end{center}
Hence, we will choose only algebras satisfying the last two identities.
Summarizing, we have the following Theorem.

\begin{theoremC}
Let ${\mathfrak n}$ be a complex  $4$-dimensional nilpotent  algebra of nil-index $3$.
Then ${\mathfrak n}$ is a $2$-step nilpotent algebra or isomorphic to one algebra from the following list:

\begin{longtable}{llllllllllllll}

$\mathfrak{L}_{01}$ & $: $&  $e_1e_2 = e_3$ & $ e_1e_3=e_4$
&  $e_2e_1 =- e_3$ & $ e_3e_1=-e_4$\\

$\mathfrak{L}_{09}$  & $: $& $e_1e_2=-e_3+e_4$ & $e_1e_3=-e_4$ &  $e_2e_1 = e_3$ &   $e_3e_1=e_4$ \\

$\mathfrak{L}_{10}$ & $: $& $e_1e_2=-e_3$ & $e_1e_3=-e_4$ &  $e_2e_1 = e_3$ &   $e_2e_2=e_4$  & $e_3e_1=e_4$   \\

$\mathfrak{L}_{11}$ & $: $ & $e_1e_1 = e_4$ &  $e_1e_2=-e_3$ & $e_1e_3=-e_4$&  $e_2e_1 = e_3$ & $e_2e_2=e_4$ & $e_3e_1=e_4$ \\

$\mathfrak{L}_{12}$ & $: $& $e_1e_1 = e_4$ & $e_1e_2=-e_3$ & $e_1e_3=-e_4$ & $e_2e_1 = e_3$ &   $e_3e_1=e_4$ \\

$\mathbb{M}_{03}^\alpha$ & $: $ & $e_1e_2=e_3$ & $e_1e_3=\alpha e_4$ & $e_2e_1=-e_3$ & \multicolumn{2}{l}{ $e_3e_1=(1-\alpha)e_4$} \\
$\mathbb{M}_{04}^\alpha$ & $: $ & $e_1e_2=e_3$ & $e_1e_3=\alpha e_4$ & $e_2e_1=-e_3$ & $e_2e_2=e_4$ & \multicolumn{2}{l}{$e_3e_1=(1-\alpha)e_4$} \\
$\mathbb{M}_{05}$ & $: $ & $e_1e_2=e_3$ & $e_2e_1=-e_3$ & $e_2e_3=e_4$ & $e_3e_1=e_4$ & $e_3e_2=-e_4$ \\
$\mathbb{M}_{06}$ & $: $ & $e_1e_2=e_3$ & $e_2e_1=-e_3$ & $e_2e_2=e_4$ & $e_2e_3=e_4$ & $e_3e_1=e_4$ & $e_3e_2=-e_4$ \\
$\mathbb{M}_{07}$ & $: $ & $e_1e_2=e_3$ & $e_2e_1=-e_3$ & $e_3e_3=e_4$ \\
$\mathbb{M}_{08}$ & $: $ & $e_1e_2=e_3+e_4$ & $e_2e_1=-e_3$ & $e_3e_3=e_4$ \\
$\mathbb{M}_{09}$ & $: $ & $e_1e_1=e_4$ & $e_1e_2=e_3$ & $e_2e_1=-e_3$ & $e_3e_3=e_4$ \\
$\mathbb{M}_{10}^\alpha$ & $: $ & $e_1e_1=\alpha e_4$ & $e_1e_2=e_3$ & $e_1e_3=e_4$ & $e_2e_1=-e_3$ & $e_3e_1=-e_4$ & $e_3e_3=e_4$ \\
$\mathbb{M}_{11}$ & $: $ & $e_1e_2=e_3+e_4$ & $e_1e_3=e_4$ & $e_2e_1=-e_3$ & $e_3e_1=-e_4$ & $e_3e_3=e_4$ \\
$\mathbb{M}_{12}^\alpha$ & $: $ & $e_1e_1=\alpha e_4$ & $e_1e_2=e_3$ & $e_1e_3=e_4$ & $e_2e_1=-e_3$ \\
& & $e_2e_2=e_4$& $e_3e_1=-e_4$ & $e_3e_3=e_4$ \\
\end{longtable}

\end{theoremC}

\section{The geometric classification of nilpotent binary and mono  Leibniz algebras}

\subsection{Definitions and notation}
Given an $n$-dimensional vector space $\mathbb V$, the set ${\rm Hom}(\mathbb V \otimes \mathbb V,\mathbb V) \cong \mathbb V^* \otimes \mathbb V^* \otimes \mathbb V$ is a vector space of dimension $n^3$. This space has the structure of the affine variety $\mathbb{C}^{n^3}.$ Indeed, let us fix a basis $e_1,\dots,e_n$ of $\mathbb V$. Then any $\mu\in {\rm Hom}(\mathbb V \otimes \mathbb V,\mathbb V)$ is determined by $n^3$ structure constants $c_{ij}^k\in\mathbb{C}$ such that
$\mu(e_i\otimes e_j)=\sum\limits_{k=1}^nc_{ij}^ke_k$. A subset of ${\rm Hom}(\mathbb V \otimes \mathbb V,\mathbb V)$ is {\it Zariski-closed} if it can be defined by a set of polynomial equations in the variables $c_{ij}^k$ ($1\le i,j,k\le n$).

Let $T$ be a set of polynomial identities.
The set of algebra structures on $\mathbb V$ satisfying polynomial identities from $T$ forms a Zariski-closed subset of the variety ${\rm Hom}(\mathbb V \otimes \mathbb V,\mathbb V)$. We denote this subset by $\mathbb{L}(T)$.
The general linear group ${\rm GL}(\mathbb V)$ acts on $\mathbb{L}(T)$ by conjugations:
$$ (g * \mu )(x\otimes y) = g\mu(g^{-1}x\otimes g^{-1}y)$$
for $x,y\in \mathbb V$, $\mu\in \mathbb{L}(T)\subset {\rm Hom}(\mathbb V \otimes\mathbb V, \mathbb V)$ and $g\in {\rm GL}(\mathbb V)$.
Thus, $\mathbb{L}(T)$ is decomposed into ${\rm GL}(\mathbb V)$-orbits that correspond to the isomorphism classes of algebras.
Let $O(\mu)$ denote the orbit of $\mu\in\mathbb{L}(T)$ under the action of ${\rm GL}(\mathbb V)$ and $\overline{O(\mu)}$ denote the Zariski closure of $O(\mu)$.

Let $\bf A$ and $\bf B$ be two $n$-dimensional algebras satisfying the identities from $T$, and let $\mu,\lambda \in \mathbb{L}(T)$ represent $\bf A$ and $\bf B$, respectively.
We say that $\bf A$ degenerates to $\bf B$ and write $\bf A\to \bf B$ if $\lambda\in\overline{O(\mu)}$.
Note that in this case we have $\overline{O(\lambda)}\subset\overline{O(\mu)}$. Hence, the definition of degeneration does not depend on the choice of $\mu$ and $\lambda$. If $\bf A\not\cong \bf B$, then the assertion $\bf A\to \bf B$ is called a {\it proper degeneration}. We write $\bf A\not\to \bf B$ if $\lambda\not\in\overline{O(\mu)}$.

Let $\bf A$ be represented by $\mu\in\mathbb{L}(T)$. Then  $\bf A$ is  {\it rigid} in $\mathbb{L}(T)$ if $O(\mu)$ is an open subset of $\mathbb{L}(T)$.
 Recall that a subset of a variety is called irreducible if it cannot be represented as a union of two non-trivial closed subsets.
 A maximal irreducible closed subset of a variety is called an {\it irreducible component}.
It is well known that any affine variety can be represented as a finite union of its irreducible components in a unique way.
The algebra $\bf A$ is rigid in $\mathbb{L}(T)$ if and only if $\overline{O(\mu)}$ is an irreducible component of $\mathbb{L}(T)$.

\subsection{Method of the description of  degenerations of algebras}

In the present work we use the methods applied to Lie algebras in \cite{GRH,GRH2}.
First of all, if $\bf A\to \bf B$ and $\bf A\not\cong \bf B$, then $\mathfrak{Der}(\bf A)<\mathfrak{Der}(\bf B)$, where $\mathfrak{Der}(\bf A)$ is the Lie algebra of derivations of $\bf A$. We compute the dimensions of algebras of derivations and check the assertion $\bf A\to \bf B$ only for such $\bf A$ and $\bf B$ that $\mathfrak{Der}(\bf A)<\mathfrak{Der}(\bf B)$.

To prove degenerations, we construct families of matrices parametrized by $t$. Namely, let $\bf A$ and $\bf B$ be two algebras represented by the structures $\mu$ and $\lambda$ from $\mathbb{L}(T)$ respectively. Let $e_1,\dots, e_n$ be a basis of $\mathbb  V$ and $c_{ij}^k$ ($1\le i,j,k\le n$) be the structure constants of $\lambda$ in this basis. If there exist $a_i^j(t)\in\mathbb{C}$ ($1\le i,j\le n$, $t\in\mathbb{C}^*$) such that $E_i^t=\sum\limits_{j=1}^na_i^j(t)e_j$ ($1\le i\le n$) form a basis of $\mathbb V$ for any $t\in\mathbb{C}^*$, and the structure constants of $\mu$ in the basis $E_1^t,\dots, E_n^t$ are such rational functions $c_{ij}^k(t)\in\mathbb{C}[t]$ that $c_{ij}^k(0)=c_{ij}^k$, then $\bf A\to \bf B$.
In this case  $E_1^t,\dots, E_n^t$ is called a {\it parametrized basis} for $\bf A\to \bf B$.
To simplify our equations, we will use the notation $A_i=\langle e_i,\dots,e_n\rangle,\ i=1,\ldots,n$ and write simply $A_pA_q\subset A_r$ instead of $c_{ij}^k=0$ ($i\geq p$, $j\geq q$, $k< r$).

%If the number of orbits under the action of $GL(\mathbb V)$ on  $\mathbb{L}(T)$ is finite, then the constructions of some %degenerations and some non-degenerations give the description of all rigid algebras and irreducible components.
Since the varieties of $4$-dimensional nilpotent mono Leibniz and $5$-dimensional nilpotent binary  Leibniz algebras  contain infinitely many non-isomorphic algebras, we have to do some additional work.
Let ${\bf A}(*):=\{ {\bf A}(\alpha)\}_{\alpha\in I}$ be a series of algebras, and let $\bf B$ be another algebra. Suppose that for $\alpha\in I$, $\bf A(\alpha)$ is represented by the structure $\mu(\alpha)\in\mathbb{L}(T)$ and $B\in\mathbb{L}(T)$ is represented by the structure $\lambda$. Then we say that $\bf A(*)\to \bf B$ if $\lambda\in\overline{\{O(\mu(\alpha))\}_{\alpha\in I}}$, and $\bf A(*)\not\to \bf B$ if $\lambda\not\in\overline{\{O(\mu(\alpha))\}_{\alpha\in I}}$.

Let $\bf A(*)$, $\bf B$, $\mu(\alpha)$ ($\alpha\in I$) and $\lambda$ be as above. To prove $\bf A(*)\to \bf B$ it is enough to construct a family of pairs $(f(t), g(t))$ parametrized by $t\in\mathbb{C}^*$, where $f(t)\in I$ and $g(t)\in {\rm GL}(\mathbb V)$. Namely, let $e_1,\dots, e_n$ be a basis of $\mathbb V$ and $c_{ij}^k$ ($1\le i,j,k\le n$) be the structure constants of $\lambda$ in this basis. If we construct $a_i^j:\mathbb{C}^*\to \mathbb{C}$ ($1\le i,j\le n$) and $f: \mathbb{C}^* \to I$ such that $E_i^t=\sum\limits_{j=1}^na_i^j(t)e_j$ ($1\le i\le n$) form a basis of $\mathbb V$ for any  $t\in\mathbb{C}^*$, and the structure constants of $\mu_{f(t)}$ in the basis $E_1^t,\dots, E_n^t$ are such rational functions $c_{ij}^k(t)\in\mathbb{C}[t]$ that $c_{ij}^k(0)=c_{ij}^k$, then $\bf A(*)\to \bf B$. In this case  $E_1^t,\dots, E_n^t$ and $f(t)$ are called a parametrized basis and a {\it parametrized index} for $\bf A(*)\to \bf B$, respectively.

We now explain how to prove $\bf A(*)\not\to\mathcal  B$.
Note that if $\mathfrak{Der} \ \bf A(\alpha)  > \mathfrak{Der} \  \bf B$ for all $\alpha\in I$ then $\bf A(*)\not\to\bf B$.
One can also use the following  Lemma, whose proof is the same as the proof of Lemma 1.5 from \cite{GRH}.

\begin{lemma}\label{gmain}
Let $\mathfrak{B}$ be a Borel subgroup of ${\rm GL}(\mathbb V)$ and $\mathcal{R}\subset \mathbb{L}(T)$ be a $\mathfrak{B}$-stable closed subset.
If $\bf A(*) \to \bf B$ and for any $\alpha\in I$ the algebra $\bf A(\alpha)$ can be represented by a structure $\mu(\alpha)\in\mathcal{R}$, then there is $\lambda\in \mathcal{R}$ representing $\bf B$.
\end{lemma}

\subsection{The geometric classification of $5$-dimensional nilpotent binary
  Leibniz algebras}
The main result of the present section is the following theorem.

\begin{theoremD}\label{geo3}
The variety of $5$-dimensional nilpotent binary  Leibniz     algebras  has
dimension  $24$   and it has  $10$  irreducible components defined by
\begin{center}
$\mathcal{C}_1=\overline{\{\mathcal{O}({\mathfrak V}_{4+1})\}},$ \
$\mathcal{C}_2=\overline{\{\mathcal{O}({\mathfrak V}_{3+2})\}},$ \
$\mathcal{C}_3=\overline{\{\mathcal{O}(\mathbb{S}_{21}^{\alpha,\beta})\}},$    \
$\mathcal{C}_4=\overline{\{\mathcal{O}(\mathbb{S}_{41}^{\alpha})\}},$   \
$\mathcal{C}_5=\overline{\{\mathcal{O}(\mathbb{L}_{28}^{\alpha})\}},$ \
$\mathcal{C}_6=\overline{\{\mathcal{O}(\mathbb{L}_{47}^{\alpha, \beta})\}},$ \
$\mathcal{C}_7=\overline{\{\mathcal{O}(\mathbb{L}_{52}^{\alpha, \beta})\}},$  \
$\mathcal{C}_8=\overline{\{\mathcal{O}(\mathbb{L}_{79}^{\alpha})\}},$   \
$\mathcal{C}_{9}=\overline{\{\mathcal{O}(\mathbb{L}_{82}))\}},$ \
$\mathcal{C}_{10}=\overline{\{\mathcal{O}({\bf B}_{09}^{\alpha})\}},$
\end{center}

In particular, there is only one  rigid algebra in this variety.

\end{theoremD}

\begin{proof}
Thanks to \cite{leib5} the variety of $5$-dimensional  nilpotent  Leibniz algebras has only $10$  irreducible components defined by

{\small
\begin{longtable}{lllllll}
${\mathfrak V}_{4+1}$ & $:$&
$e_1e_2=e_5$& $e_2e_1=\lambda e_5$ &$e_3e_4=e_5$&$e_4e_3=\mu e_5$\\

${\mathfrak V}_{3+2}$ &$ :$&
$e_1e_1 =  e_4$& $e_1e_2 = \mu_1 e_5$ & $e_1e_3 =\mu_2 e_5$&
$e_2e_1 = \mu_3 e_5$  & $e_2e_2 = \mu_4 e_5$  \\
& & $e_2e_3 = \mu_5 e_5$  & $e_3e_1 = \mu_6 e_5$  & \multicolumn{2}{l}{$e_3e_2 = \lambda e_4+ \mu_7 e_5$ } & $e_3e_3 =  e_5$  \\

$\mathbb{S}_{21}^{\alpha,\beta}$ &$:$&
$e_1e_1=\alpha e_5$ &  $e_1 e_2 =e_3+e_4+\beta e_5$  & $e_1e_3=e_5$ & $e_2 e_1 =-e_3$ \\&& $e_2e_2=e_5$  & $e_2e_3 =e_4$ & $e_3e_1=-e_5$ & $e_3e_2=-e_4$\\

$\mathbb{S}_{22}^{\alpha}$ &$:$& $e_1e_1=e_5$ &  $e_1 e_2 =e_3$  & $e_1e_3=e_5$ & $e_2 e_1 =-e_3$ \\&& $e_2e_2=\alpha e_5$  &  $e_2e_4=e_5$ & $e_3e_1=-e_5$ & $e_4e_4=e_5$\\

$\mathbb{S}_{41}^{\alpha}$ &$:$& $e_1e_1=e_5$ & $e_1 e_2 =e_3$  & $e_1e_3=e_5$ & $e_2 e_1 =-e_3$ & $e_2e_2=\alpha e_5$ \\&& $e_2e_3=e_4$ & $e_2e_4=e_5$ &  $e_3e_1=-e_5$ & $e_3e_2=-e_4$ & $e_4e_2=-e_5$\\

$\mathbb{L}_{28}^{\alpha}$ & $: $ & $e_1e_1=e_3$  & $e_1e_2=e_3$ & $e_1e_4=\alpha e_5$ &  $e_2e_2=e_5$ \\
& & $e_3e_1=e_5$ & $e_3e_2=e_5$ & $e_4e_4=e_5$ \\

$\mathbb{L}_{47}^{\alpha, \beta}$ & $: $ & $e_1e_1=e_3$  & $e_1e_2=e_4$ & $e_2e_1=-\alpha e_3$ \\
&&  $e_2e_2=-e_4$ & $e_3e_1=e_5$ & $e_4e_2=\beta e_5$ &  \\

$\mathbb{L}_{52}^{\alpha, \beta}$ & $: $ & $e_1e_2=e_3$ & $e_1e_3=-e_5$ & $e_1e_4=e_5$ & $e_2e_1=e_4$ &  $e_2e_3=\beta e_5$  \\
&& $e_2e_4=-\beta e_5$ & $e_3e_1=e_5$ & $e_3e_2=e_5$ &  $e_4e_1=\alpha e_5$&  $e_4e_2=\beta e_5$    \\

$\mathbb{L}_{79}^{\alpha}$ & $: $  & $e_1e_1=e_3$ & $e_1e_2=e_4$ & $e_2e_1=e_3$ \\
& & $e_2e_2=e_4+e_5$ &  $e_3e_1=e_4+\alpha e_5$ & $e_3e_2=e_5$ & $e_4e_1=e_5$ \\

$\mathbb{L}_{82}$ & $: $  & $e_1e_1=e_2$ &  $e_2e_1=e_3$ & $e_3e_1=e_4$ & $e_4e_1=e_5$\\
\end{longtable}}

After carefully  checking  the dimensions of orbit closures of the more important for us algebras, we have

\begin{longtable}{rcl}

$\dim  \mathcal{O}({\mathfrak V}_{3+2})$&$=$&$24,$ \\
$\dim \mathcal{O}(\mathbb{L}_{28}^{\alpha})=
\dim \mathcal{O}(\mathbb{L}_{47}^{\alpha, \beta})=
\dim \mathcal{O}(\mathbb{L}_{52}^{\alpha, \beta})=
\dim \mathcal{O}(\mathbb{L}_{79}^{\alpha})=\dim \mathcal{O}({\bf B}_{09}^{\alpha})$&$=$&$22,$ \\

$\dim \mathcal{O}(\mathbb{S}_{21}^{\alpha,\beta})$&$=$&$ 21,$\\

$\dim \mathcal{O}(\mathbb{S}_{41}^{\alpha})=
\dim \mathcal{O}({\mathfrak V}_{4+1})=
\dim \mathcal{O}(\mathbb{L}_{82})$&$=$&$20.$

\end{longtable}

 Hence,
$\mathbb{L}_{28}^{\alpha}, \mathbb{L}_{47}^{\alpha,\beta}, \mathbb{L}_{52}^{\alpha}, \mathbb{L}_{79}^{\alpha}$, ${\mathfrak V}_{3+2}$ and ${\bf B}_{09}^{\alpha}$ give $6$ irreducible components. Moreover,  the algebra $\mathbb{L}_{82}$ is a  only one one-generated $5$-dimensional
nilpotent Leibniz algebra. It is known that if $A(*) \rightarrow B$ and $B$ is one-generated, then $A$ is one-generated, then $\mathbb{L}_{82}$ is rigid and its closure gives us an irreducible component.

Below we have listed all the important reasons for necessary non-degenerations.

\begin{longtable}{lcl|l}
\hline
    \multicolumn{4}{c}{Non-degenerations reasons} \\
\hline

${\bf B}_{09}^{\alpha} $&$  \not \rightarrow  $ & ${\mathfrak V}_{4+1}, \mathbb{S}_{41}^{\alpha},
\mathbb{S}_{21}^{\alpha,\beta}$  &
$\mathcal R=\left\{\begin{array}{ll}
\mbox{new basis} \ f_1=e_1, \ f_2=e_2, \ f_3=e_4, \ f_4=e_3, \ f_5=e_5\\
A_1^2 \subseteq  A_4, \ A_1A_5+A_5A_1=0, \\
c_{41}^5=-c_{14}^5, \quad c_{42}^5=-c_{24}^5, \quad
c_{43}^5=-c_{34}^5, \quad  c_{44}^5=0\\
%c_{22}^5c_{43}^5c_{43}^5 = c_{42}^5 (c_{23}^5 c_{43}^5 + c_{32}^5c_{43}^5-c_{33}^5 c_{42}^5) \\
\end{array}\right\}
$\\

\hline

\end{longtable}

The rest of the degenerations is given in the following two tables and it completes the proof of the Theorem.
 {\footnotesize
\begin{longtable}{lcl lll} \hline

${\bf B}_{09}^{\frac{4\alpha (1+t^2)-1}{4\alpha+t^2-1}}$&$\to$&$\mathbb{S}_{22}^{\alpha}$ &
\multicolumn{3}{l}{$E_1^t=\frac{t(1+t^2)}{t^2+\sqrt{1-4\alpha (1+t^2)}}e_1+\frac{t^3+t\sqrt{1-4\alpha (1+t^2)}}{1-4\alpha -t^2}e_2+\frac{t\sqrt{1-4\alpha (1+t^2)}}{1-4\alpha -t^2}e_3-\frac{t}{\sqrt{1-4\alpha (1+t^2)}}e_4$} \\

\multicolumn{6}{l}{$E_2^t=\frac{t^2(1-\sqrt{1-4\alpha (1+t^2)})}{2t^2+2\sqrt{1-4\alpha (1+t^2)}}e_1+\frac{t^2(1-4\alpha +\sqrt{1-4\alpha (1+t^2)})}{2(1-4\alpha-t^2)}e_2+\frac{t^2\sqrt{1-4\alpha (1+t^2)}}{1-4\alpha-t^2}e_3$} \\

\multicolumn{4}{l}{$E_3^t=\frac{t^3\sqrt{1-4\alpha (1+t^2)}}{1-4\alpha-t^2}e_3+\frac{t^3(1+t^2)(1-4\alpha+\sqrt{1-4\alpha (1+t^2))}}{2(1-4\alpha-t^2)(t^2+\sqrt{1-4\alpha (1+t^2)})}e_5$} & $E_4^t=\frac{t^2}{\sqrt{1-4\alpha (1+t^2)}}e_4$ & $E_5^t=\frac{t^4}{1-4\alpha-t^2}e_5$  \\

\hline

\end{longtable} }

 For the rest of degenerations, in  case of  $E_1^t,\dots, E_5^t$ is a {\it parametric basis} for ${\bf A}\to {\bf B},$ it will be denoted by
${\bf A}\xrightarrow{(E_1^t,\dots, E_5^t)} {\bf B}$.

{\small
\begin{longtable}{lcl|lcl} \hline

${\bf B}_{02}$ & $\xrightarrow{ (t^{-1} e_1, t^{-1}e_2, t^{-2}e_3, t^{-1}e_4, t^{-3}e_5)}$ & ${\bf B}_{01}$ &

${\bf B}_{04}$ & $\xrightarrow{ (t^{-1}e_1, t^{-1}e_2,  t^{-2}e_3, t^{-2}e_4, t^{-4}e_5)}$ & ${\bf B}_{02}$
\\  \hline

${\bf B}_{04}$ & $\xrightarrow{ (t^{-2} e_1, e_2, t^{-2}e_3, t^{-1}e_4, t^{-3}e_5)}$ & ${\bf B}_{03}$ &

${\bf B}_{08}$ & $\xrightarrow{ (t^{-1}e_1, e_2, t^{-1}e_3, t^{-1}e_4, t^{-2}e_5)}$ & ${\bf B}_{04}$
\\  \hline

${\bf B}_{08}$ & $\xrightarrow{ (t^{-1}e_1, t^{-1}e_2, t^{-2}e_3, e_4, t^{-2}e_5)}$ & ${\bf B}_{05}$
&

${\bf B}_{08}$ & $\xrightarrow{ (t^{-2}e_1, e_2, t^{-2}e_3, e_4, t^{-2}e_5)}$ & ${\bf B}_{06}$

\\  \hline

${\bf B}_{09}^{\frac{1}{t^2}}$ & $\xrightarrow{ (t^{-1}e_1, t^{-1}e_2, t^{-2}e_3, e_4, t^{-2}e_5)}$ & ${\bf B}_{07}$  &

${\bf B}_{09}^{\frac{1}{t}}$ & $\xrightarrow{ (t^{-1}e_1, e_2, t^{-1}e_3, e_4, t^{-1}e_5)}$ & ${\bf B}_{08}$ \\  \hline

${\bf B}_{11}$ & $\xrightarrow{ (te_1, t^{-1}e_2, e_3, t^2e_4, t^{2}e_5)}$ & ${\bf B}_{10}$  &

${\bf B}_{14}$ & $\xrightarrow{ (t^{-1}e_1, e_2, t^{-1}e_3, t^{-1}e_4, t^{-2}e_5)}$ & ${\bf B}_{11}$  \\

\hline

\multicolumn{6}{c}{${\bf B}_{09}^{\frac{\alpha}{t}}\quad \xrightarrow{ (t^{-1}e_1+t^{-2}e_2+t^{-1}e_3, t^{-1}e_2, t^{-2}e_3, t^{-1}e_4, t^{-3}e_5)} \quad {\bf B}_{12}^{\alpha}$}

\\
\hline

${\bf B}_{09}^{0}$ & $\xrightarrow{ (e_1-e_2, -te_2, -te_3, t^{-1}e_4, -e_5)}$ & ${\bf B}_{13}$
&

${\bf B}_{09}^{-t^2}$ &  $\xrightarrow{ (e_1-e_2, -te_2, -te_3, t^{-1}e_4, -e_5)}$ &  ${\bf B}_{14}$

\\  \hline

\end{longtable}
}

\end{proof}

It is easy to see, that each one-generated binary Leibniz algebra is a Leibniz algebra.
On the other side, there is only one one-generated $n$-dimensional Leibniz algebra.
Let us also give a trivial observation.

\begin{lemma}
 The variety of $n$-dimensional nilpotent binary Leibniz algebras has at least one rigid algebra.
\end{lemma}

The present lemma raises a question.

\begin{openq}
    Are there non-one-generated rigid algebras in the variety of $n$-dimensional nilpotent binary Leibniz algebras?
\end{openq}

\subsection{The geometric classification of $4$-dimensional nilpotent  algebras  of nil-index 3}
The main result of the present section is the following theorem.

\begin{theoremE}\label{geo4}
The variety of complex $4$-dimensional nilpotent  algebras of nil-index $3$ has
dimension {\it 15 }  and it has
two   irreducible components
\begin{center}
$\mathcal{C}_1=\overline{\{\mathcal{O}(\mathfrak{N}_2(\gamma ))\}},$ \
$\mathcal{C}_2=\overline{\{\mathcal{O}(\mathbb{M}_{12}^{\alpha})\}}.$

\end{center}

In particular, there are no rigid algebras in this variety.

\end{theoremE}

\begin{proof}
Thanks to \cite{kppv},
the variety of $2$-step nilpotent algebras has two irreducible components defined by the following families of algebras

\begin{longtable}{lllllll}
$\mathfrak{N}_2(\gamma )$
& $e_1e_1 = e_3$ &$e_1e_2 = e_4$ & $  e_2e_1 = -\gamma e_3$ & $ e_2e_2 = -e_4$\\

$\mathfrak{N}_3(\alpha    )$ &
$e_1e_1 = e_4$& $e_1e_2 = \alpha e_4$ &
$e_2e_1 = -\alpha e_4$ & $e_2e_2 = e_4$ &$ e_3e_3 = e_4$\\
\end{longtable}

Thanks to \cite{kppv},
we have $\mathfrak{L}_{11} \to \mathfrak{L}_{01}, \mathfrak{L}_{09}, \mathfrak{L}_{10},
\mathfrak{L}_{12}.$
In the following tables, we have listed all
necessary degenerations.

{\small
\begin{longtable}{lclll} \hline

%$\mathbb{M}_{06}$&$\to$&$\mathfrak{L}_{09}$ &
%$E_1^t=t^2e_1-te_2-e_3$  &
%$E_2^t=-t^4e_2+(t^5-2t^3)e_3$ \\
%&&& $E_3^t=t^6e_3+(t^5-t^4)(t+1)^2e_4$& $E_4^t=t^7e_4$\\

%\hline

%$\mathbb{M}_{06}$&$\to$&$\mathfrak{L}_{10}$ &
%$E_1^t=te_1-e_2-t^{-1}e_3$  &
%$E_2^t=-te_2-2e_3$ \\
%&&&$E_3^t=t^3e_3-(t+1)e_4$& $E_4^t=t^2e_4$\\

%\hline

$\mathbb{M}_{06}$&$\to$&$\mathfrak{L}_{11}$ &
$E_1^t=te_1-e_2+\frac{t^2-1}{t}e_3$  &
$E_2^t=-te_2-2e_3$ \\
&&&$E_3^t=t^2e_3-(t^2+t+1)e_4  $&$ E_4^t=t^2e_4$\\

%\hline

%$\mathbb{M}_{06}$&$\to$&$\mathfrak{L}_{12}$ &
%$E_1^t=e_1-\frac{1}{t}e_2-\frac{t+1}{t^2}e_3$  &
%$E_2^t=e_2+\frac{2}{t}e_3$ \\
%&&&$E_3^t=-e_3+\frac{1}{t^2}e_4$&$E_4^t=-\frac{1}{t}e_4$\\

\hline

$\mathbb{M}_{06}$&$\to$&$ \mathbb{M}_{03}^{\alpha}$ &
$E_1^t=e_1+\alpha e_2-\alpha^{2} e_3$  &
$E_2^t=te_2-2\alpha te_3$ \\
&&& $E_3^t=te_3-(\alpha^2-\alpha)te_4  $&$ E_4^t=t e_4$\\

\hline

$\mathbb{M}_{06}$&$\to $&$\mathbb{M}_{04}^{\alpha}$ &
$E_1^t=t e_1+\alpha t e_2-\alpha^{2}te_3$  &
$E_2^t=t^2e_2-2\alpha t^2e_3$ \\
&&&$E_3^t=t^3e_3-(\alpha^2-\alpha)t^3e_4  $&$ E_4^t=t^4 e_4$\\

\hline

$\mathbb{M}_{12}^{\frac{t}{1-t}}$&$\to$&$ \mathbb{M}_{06}$ &
$E_1^t=\frac{i t}{\sqrt{1-t}}e_1+\frac{t}{t-1}e_2-\frac{i t}{\sqrt{1-t}}e_3$  &
$E_2^t=-\frac{2i t}{\sqrt{1-t}}e_1-\frac{2t^2}{t-1}e_2$ \\
&&&$E_3^t=-\frac{2i t^2}{\sqrt{1-t}}e_3-\frac{2t^2}{t-1}e_4 $&$  E_4^t=\frac{4t^3}{t-1}e_4$\\
\hline

\end{longtable}}

 For the rest of degenerations, in  case of  $E_1^t,\dots, E_4^t$ is a {\it parametric basis} for ${\bf A}\to {\bf B},$ it will be denoted by
${\bf A}\xrightarrow{(E_1^t,\dots, E_4^t)} {\bf B}$.

\begin{longtable}{lcl | lcl} \hline

$\mathbb{M}_{12}^{\frac{1}{t^2}} $ &
$ \xrightarrow{( te_1, e_2, e_3-\alpha t^{-1}e_4, e_4)} $ &
$ \mathfrak{N}_3(\alpha)$ &

$\mathbb{M}_{06}$ &
$\xrightarrow{(t^{-1}e_1, t^{-1}e_2, t^{-2}e_3,  t^{-3}e_4)}$
& $\mathbb{M}_{05}$ \\\hline

$\mathbb{M}_{11}$ &
$\xrightarrow{(e_1-e_2+e_3, t^{-1}e_2, t^{-1}e_3,  t^{-2}e_4)}$ & $ \mathbb{M}_{07}$ &

$\mathbb{M}_{11}$&
$\xrightarrow{(te_1-te_2+te_3, t^{-1}e_2, e_3, e_4)}$ & $ \mathbb{M}_{08}$ \\\hline

$\mathbb{M}_{11}$ &
$\xrightarrow{(e_1+(t^{-2}-1)e_2+e_3, t^{-1}e_2, t^{-1}e_3, t^{-2}e_4)}$ & $ \mathbb{M}_{09}$&

$\mathbb{M}_{11}$&
$\xrightarrow{(t^{-1}e_1+\alpha t^{-1}e_2, e_2, t^{-1}e_3, t^{-2}e_4)}$ &$ \mathbb{M}_{10}^{\alpha}$ \\\hline

\multicolumn{6}{c}{$\mathbb{M}_{12}^{-\frac{1}{4t^2}}$ \quad
$\xrightarrow{(t^{-1}e_1+\frac{1}{2}t^{-1}e_2, e_2, t^{-1}e_3-\frac{1}{2}t^{-2}e_4, t^{-2}e_4)}$ \quad $ \mathbb{M}_{11}$}\\

\hline

\end{longtable}

After carefully  checking  the dimensions of orbit closures of the rest of the algebras, we have
\begin{center}
$\dim \mathcal{O}(\mathfrak{N}_2(\gamma))= 12,$
$\dim \mathcal{O}(\mathbb{M}_{12}^{\alpha})=15.$
\end{center}

Non-degenerations reasons are given below.
\begin{longtable}{lcl|l}
\hline
${\mathbb M}_{12}^{\alpha} $&$  \not \rightarrow  $ & $ \mathfrak{N}_2(\gamma)$ &
$\mathcal R=\left\{\begin{array}{ll}
A_1^2  \subset A_3, \
c_{11}^3=c_{22}^3=0, \ c_{12}^3=-c_{21}^3  \\
\end{array}\right\}
$\\

\hline
\end{longtable}
 Hence,
$\mathfrak{N}_2(\gamma)$ and $\mathbb{M}_{12}^{\alpha}$ give  irreducible components.
 \end{proof}

Let us remember that the variety of $n$-dimensional nilpotent algebras of nil-index 2 is irreducible (see, \cite{KW01}).
The last theorem gives an example of
a variety of nilpotent algebras of nil-index $3,$
that has two irreducible components.
The last observation is motivating the following question.

\begin{openq}
Find  the number of irreducible components of
the variety of $n$-dimensional nilpotent algebras of nil-index $3.$
\end{openq}

\subsection{The geometric classification of $4$-dimensional nilpotent mono Leibniz algebras}
The main result of the present section is the following theorem.

\begin{theoremF}
The variety of complex $4$-dimensional nilpotent mono Leibniz algebras has
dimension {\it 15 }  and it has
3  irreducible components:
\begin{center}
$\mathcal{C}_1=\overline{\{\mathcal{O}(\mathfrak{L}_{2})\}},$ \
$\mathcal{C}_2=\overline{\{\mathcal{O}(\mathbb{M}_{12}^{\alpha})\}},$ \
$\mathcal{C}_3=\overline{\{\mathcal{O}(\mathbb{M}_{14}^{\alpha})\}}.$

\end{center}
In particular, there is only one rigid algebra in this variety.
\end{theoremF}

\begin{proof}
Thanks to \cite{HS86} the variety of $4$-dimensional  nilpotent  Leibniz algebras has only $4$  irreducible components defined by

\begin{longtable}{llllllll}
$\mathfrak{N}_3(\alpha)$ & $:$&
$e_1e_1=e_4$ & $e_1e_2=\alpha e_4$ &$e_2e_1=-\alpha e_4$ & $e_2e_2=e_4$ & $e_3e_3=e_4$\\

$\mathfrak{L}_2$ &$ :$ &
$e_1e_1 = e_2$ & $e_2e_1 =e_3$ & $e_3e_1 = e_4$ \\

$\mathfrak{L}_5$ &$:$&
$e_1e_1=e_3$ &  $e_2e_1=e_3$ & $e_2e_2=e_4$ & $e_3 e_1 =e_4$ \\

$\mathfrak{L}_{11}$ &$:$& $e_1e_1=e_4$ &  $e_1 e_2 =-e_3$ & $e_1e_3=-e_4$ & $e_2 e_1 =e_3$ & $e_2e_2=e_4$  &  $e_3e_1=e_4$ \\
\end{longtable}

In the Theorem E,  variety of complex $4$-dimensional nilpotent  algebras of nil-index $3$ has one irreducible component

\begin{longtable}{llllllllllllll}
$\mathbb{M}_{12}^\alpha$ & $: $ & $e_1e_1=\alpha e_4$ & $e_1e_2=e_3$ & $e_1e_3=e_4$ & $e_2e_1=-e_3$ \\
& & $e_2e_2=e_4$& $e_3e_1=-e_4$ & $e_3e_3=e_4$ \\
\end{longtable}

From Theorem E, we obtain that the algebra $\mathbb{M}_{12}^\alpha$ is degeneration to the algebra $\mathfrak{L}_{11}$.

After carefully  checking  the dimensions of orbit closures of the more important for us algebras, we have

\begin{center}
$\dim  \mathcal{O}(\mathfrak{L}_2)=12,$ \,
$\dim \mathcal{O}(\mathbb{M}_{12}^{\alpha})=
\dim \mathcal{O}(\mathbb{M}_{14}^{\alpha})=15.$
\end{center}

Hence,
$\mathbb{M}_{12}^{\alpha}$ and $\mathbb{M}_{14}^{\alpha}$ give $2$ irreducible components. Moreover, since the algebra $\mathfrak{L}_2$ is only one one-generated $4$-dimensional nilpotent Leibniz algebra, then $\mathfrak{L}_2$ is rigid and its closure gives us an irreducible component.

The rest of the degenerations is given below in the following tables and it completes the proof of the Theorem:

{\small
\begin{longtable}{lcl ll} \hline

$\mathbb{M}_{14}^{t+1}$&$\to$&$\mathfrak{L}_5$ &  $E_1^t=-2t^2e_1+(t^2-t)e_2+t^2e_3$ &
$E_2^t=(t^2-t^3)e_2-(2t^3+t^2)e_3$  \\
&&& $E_3^t=2t^4e_3-(t^4-t^3)e_4$ & $E_4^t=2(t^6-t^5)e_4$\\

 \hline

$\mathbb{M}_{14}^{\alpha}$ & $\to$ & $\mathbb{M}_{16}^{\alpha}$ &
\multicolumn{2}{l}{[let us fix $x=t(\sqrt{1-4\alpha}-t)$]}\\

\multicolumn{5}{l}{$E_1^t=x e_1-\frac{2x}{1-\sqrt{1-4(\alpha+x)}}e_2-\frac{x(2\alpha(\alpha+x)+x(1+\sqrt{1-4(\alpha+x)}))}{2(\alpha +x)^2}e_3$} \\

\multicolumn{5}{l}{$E_2^t=-\frac{2x\sqrt{1-4(\alpha+x)}}{1-\sqrt{1-4(\alpha+x)}}e_2-\frac{x}{2}(1+\sqrt{1-4(\alpha+x)})e_3$} \\

\multicolumn{5}{l}{$E_3^t=\frac{x^3(2(\alpha+x)-1-\sqrt{1-4(\alpha +x)})}{2(\alpha+x)^2}e_3-\frac{x^3(2(\alpha+x)-1-\sqrt{1-4(\alpha +x)})(1+\sqrt{1-4(\alpha +x)})}{4(\alpha+x)^3}e_4$} \\

\multicolumn{5}{l}{$E_4^t=\frac{x^4(1+\sqrt{1-4(\alpha +x)-2(\alpha+x)})(1+\sqrt{1-4(\alpha+x)})}{4(\alpha+x)^3}e_4$} \\

\hline

$\mathbb{M}_{14}^{\alpha\neq 0}$ & $\to$ & $\mathbb{M}_{18}^{\alpha\neq 0}$ &
\multicolumn{2}{l}{[let us fix $x=-t(\sqrt{1-4\alpha}+t)$]}\\

\multicolumn{5}{l}{$E_1^t=x e_1-\frac{2x}{1-\sqrt{1-4(\alpha+x)}}e_2-\frac{x(2\alpha(\alpha+x)+x(1+\sqrt{1-4(\alpha+x)}))}{2(\alpha +x)^2}e_3$} \\

\multicolumn{5}{l}{$E_2^t=-\frac{2x\sqrt{1-4(\alpha+x)}}{1-\sqrt{1-4(\alpha+x)}}e_2-\frac{x}{2}(1+\sqrt{1-4(\alpha+x)})e_3$} \\

\multicolumn{5}{l}{$E_3^t=\frac{x^3(2(\alpha+x)-1-\sqrt{1-4(\alpha +x)})}{2(\alpha+x)^2}e_3-\frac{x^3(2(\alpha+x)-1-\sqrt{1-4(\alpha +x)})(1+\sqrt{1-4(\alpha +x)})}{4(\alpha+x)^3}e_4$} \\

\multicolumn{5}{l}{$E_4^t=\frac{x^4(1+\sqrt{1-4(\alpha +x)-2(\alpha+x)})(1+\sqrt{1-4(\alpha+x)})}{4(\alpha+x)^3}e_4$} \\

\hline
$\mathbb{M}_{20}$&$\to$&$\mathbb{M}_{22}$ &
$E_1^t=t\sqrt{1-t}e_1+i (t^2-t)e_2-i (t^3-t)e_3$  &
$E_2^t=i t\sqrt{1-t}e_1+te_2+(t^2-t)e_3$ \\
&&& $E_3^t=t^3e_3+i t^3\sqrt{1-t}e_4$ &$ E_4^t=t^4\sqrt{1-t}e_4$\\

\hline
\end{longtable} }
and

\begin{longtable}{lcl|lcl} \hline

$\mathbb{M}_{02}$ & $ \xrightarrow{({t}^{-1}e_1,
t^{-2}e_2, e_3, t^{-2}e_4)}$ & $ \mathbb{M}_{01}$ &
$\mathbb{M}_{14}^{0}$  &
$ \xrightarrow{( e_1, e_3, te_2, te_4)}$ & $\mathbb{M}_{02}$ \\

\hline

$\mathbb{M}_{14}^{\alpha}$ &  $\xrightarrow{( t^{-1}e_1, t^{-1}e_2, t^{-2}e_3, t^{-3}e_4)}$ & $\mathbb{M}_{13}^{\alpha}$ &

$\mathbb{M}_{16}^{\alpha}$ & $\xrightarrow{( t^{-}e_1,t^{-1}e_2, t^{-2}e_3, t^{-3}e_4)}$ &  $\mathbb{M}_{15}^{\alpha}$\\

\hline

$\mathbb{M}_{18}^{\alpha\neq0}$ &
$\xrightarrow{( t^{-1}e_1, t^{-1}e_2, t^{-2}e_3,  t^{-3}e_4)}$ & $\mathbb{M}_{17}^{\alpha\neq0}$ &

$\mathbb{M}_{20}$ &
$\xrightarrow{(  t^{-1}e_1,  t^{-1}e_2,  t^{-2}e_3,  t^{-3}e_4)}$ & $\mathbb{M}_{19}$\\

\hline

$\mathbb{M}_{14}^{\frac{1}{t^2}}$ &
$\xrightarrow{( te_2, e_1, e_3, te_4)}$ & $\mathbb{M}_{20}$ &

$\mathbb{M}_{22}$ &
$\xrightarrow{( t^{-1}e_1, t^{-1}e_2, t^{-2}e_3,  t^{-3}e_4)}$ & $\mathbb{M}_{21}$\\

\hline

\end{longtable}

\end{proof}

It is easy to see, that each one-generated mono Leibniz algebra is a Leibniz algebra.
On the other side, there is only one one-generated $n$-dimensional Leibniz algebra.
Let us also give a trivial observation.

\begin{lemma}
 The variety of $n$-dimensional nilpotent mono Leibniz algebras has at least one rigid algebra.
\end{lemma}

The present lemma raises a question.

\begin{openq}
    Are there non-one-generated rigid algebras in the variety of $n$-dimensional nilpotent mono Leibniz algebras?
\end{openq}

%{\bf Compliance with ethical standard}

%{\bf Author contributions}  All authors contributed to the study, conception and design. All authors read and approved the final manuscript.

%{\bf Conflict of interest}
%There is no potential conflict of ethical approval, conflict of interest, and ethical standards.

%{\bf Data Availibility} Data sharing is not applicable to this article as no datasets were generated or analyzed during the current study.

%\newpage

\end{document}